\documentclass[twoside, 11pt]{article}
\usepackage{amsmath,amssymb,amsthm,mathrsfs}
\usepackage[margin=2cm]{geometry} % showframe,
\usepackage{lipsum}
\usepackage{titlesec,hyperref}
\usepackage{fancyhdr}
\usepackage[numbers,sort&compress]{natbib}
\usepackage{color}

\fancypagestyle{plain}
{
\fancyhf{}

}

\titleformat{\subsection}{\it}{\thesubsection.\enspace}{1.5pt}{}
\titleformat{\subsubsection}{\it}{\thesubsubsection.\enspace}{1.5pt}{}

\newtheorem{theorem}{Theorem}[section]

\newtheorem{remark}{Remark}[section]
\newtheorem{lemma}[theorem]{Lemma}
\numberwithin{equation}{section}

\allowdisplaybreaks

\def\a{\alpha}

\def\ga{\gamma}
\def\si{\sigma}
\def\inf{\infty}
\def\t{\tau}
\def\ep{\varepsilon}
\def\y{(1+y)}
\def\u{u^\varepsilon}
\def\v{v^\varepsilon}
\def\w{w^\varepsilon}
\def\g{g^\varepsilon}

\def\TR{\mathbb{T}\times \mathbb{R}^+}

\def\es{\epsilon}

\def\d{\delta}

\def\p{\partial}

\def\O{\Omega}
\def\T{\mathbb{T}}

\begin{document}
\title{A Note on the Well-posedness of Prandtl Equations in Dimension Two\hspace{-4mm}}

\author{Jincheng Gao$^{\dag}$   \quad    Daiwen Huang$^{\ddag}$   \quad    Zheng-an Yao$^{\dag}$\\[10pt]
\small {$^\dag $School of Mathematics, Sun Yat-sen University,}\\
\small {Guangzhou, 510275,  P.R.China}\\[5pt]
\small {$^\ddag $Institute of Applied Physics and Computational Mathematics,}\\
\small {Beijing, 100088, P.R.China}\\[5pt]
}

\footnotetext{Email: \it gaojch5@mail.sysu.edu.cn(J.C.Gao),
                     \it hdw55@tom.com(D.W.Huang)
                     \it mcsyao@mail.sysu.edu.cn(Z.A.Yao).}
\date{}

\maketitle

\begin{abstract}
In this paper, we investigate the local-in-time well-posedness for the two-dimensional
Prandtl equations in weighted Sobolev spaces under the Oleinik's monotonicity condition.
Due to the loss of tangential derivative caused by vertical velocity
appearing in convective term, we add with artificial horizontal viscosity term
to construct an approximate system that can obtain local-in-time well-posedness results easily.
For this approximate system, we construct a new weighted norm for the vorticity
to derive a positive life time(independent of artificial viscosity coefficient)
and obtain uniform bound for vorticity in this weighted norm.
Then, based on compactness argument, we prove the solution of approximate system
converging to the solution of original Prandtl equations,
and hence, obtain the local-in-time well-posedness result for the Prandtl
equations with any large initial data, which improves
the recent work \cite{Masmoudi-Wong}.

\end{abstract}

%\tableofcontents

\section{Introduction}

Throughout this paper, we are concerned with the two-dimensional Prandtl equations,
derived by Ludwing Prandtl \cite{Prandtl},
in a periodic domain $\T\times \mathbb{R}^+
:=\{(x, y):x\in \mathbb{R}/\mathbb{Z}, ~0 \le y <+\infty\}$:
\begin{equation}\label{Prandtl-eq}
\left\{
\begin{aligned}
&\p_t u+u \p_x u+v \p_y u-\p_y^2 u+\p_x p=0,\\
&\p_x u+\p_y v=0,\\
&u|_{y=0}=v|_{y=0}=0, \quad \underset{y\rightarrow +\infty}{\lim}u=U(t,x),\\
&u|_{t=0}=u_{0},
\end{aligned}
\right.
\end{equation}
where the velocity field $(u, v):=(u(t, x, y), v(t, x, y))$ is unknown,
and the initial data $u_0:=u_0(x, y)$ and the outer flow
$U:=U(t, x)$ are given and satisfy the compatibility conditions:
\begin{equation}
\left.u_0 \right|_{y=0}=0, \quad
\lim_{y\rightarrow +\infty} u_0=\left. U \right|_{t=0}.
\end{equation}
Furthermore, the given scalar pressure $p:=p(t, x)$ and the outer flow $U$
satisfy the Bernoulli's law:
\begin{equation}\label{o-Bernoulli}
\partial_t U+U \partial_x U=-\partial_x p.
\end{equation}

Note that the Prandtl equations mentioned above arise from the vanishing viscosity
limit of Navier-Stokes equations in a domain with Dirichlet boundary condition.
This is due to the formation of a boundary layer, where the solution undergoes a sharp transition
from a solution of the Euler system to the zero non-slip boundary condition
on boundary of the Navier-Stokes system. This boundary layer satisfies the
Prandtl boundary layer equations formally.
The first systematic work in rigorous mathematics was obtained by Oleinik \cite{{Oleinik1},{Oleinik2}}
in which she established the local in time well-posedness of Prandtl equations
in dimension two by applying the Crocco transformation under the monotonicity condition
on the tangential velocity field in the normal direction to the boundary.
For more extensional mathematical results, the interested readers can refer to
the classical book finished by Oleinik and Samokhin \cite{Oleinik-Samokhin}.
In addition to Oleinik's monotonicity assumption on the velocity field,
by imposing a so-called favorable condition on the pressure, Xin and Zhang \cite{Xin-Zhang}
obtained the existence of global weak solutions to the Prandtl equation.
As mentioned in \cite{{Caflisch-Sammartino},{Varet-Dormy},{Varet-Nguyen}}, the local in time
well-posedness of Prandtl equation for the initial data in Sobolev space is an open problem.
Then, the researchers in \cite{Xu-Yang-Xu} and \cite{Masmoudi-Wong} independently used the
nonlinear cancelation method to establish well-posedness theory for the two-dimensional
Prandtl equations in the framework of Sobolev spaces.
\emph{
Note that the local-in-time well-posedness result in $H^{s,\ga}$ Sobolev framework,
obtained by Masmoudi and Wong \cite{Masmoudi-Wong},
required there exists a small constant $\d_0$ such that
\begin{equation}\label{initial-data-1}
\sum_{|\a| \le 2}|\y^{\si+\a_2} D^\a w_0|^2 \le {(2\d_0)^{-2}},
\end{equation}
and for $s=4$, they required additionally
\begin{equation}\label{initial-data-2}
\|w_0\|_{H^{s,\ga}_g}\le C\d_0^{-1}.
\end{equation}
Thus, the main target in this paper is to remove the conditions \eqref{initial-data-1}-\eqref{initial-data-2}.
In other words, the local-in-time well-posedness results for the Prandtl equations \eqref{Prandtl-eq}
will be established for any large initial data.
}

Finally, we point out that it is an outstanding open problem to rigorously justify
the validity of expansion in the inviscid limit.
On one hand, within space of functions that are analytic,
Sammartino and Caflisch \cite{{Sammartino-Caflisch1},{Sammartino-Caflisch2}}
obtained the well-posedness in the framework of analytic functions without the
monotonicity condition on the velocity field and justified the boundary layer expansion
for the unsteady incompressible Navier-Stokes equations.
On the other hand, for more ``realistic'' functional settings,
Guo and Nguyen\cite{Guo-Nguyen2} justified the boundary layer expansion for
the steady incompressible flow with a non-slip boundary condition on a moving plate.
This result was  generalized to the case of non-moving or with external forcing(cf.\cite{{Guo-Iyer1},{Guo-Iyer2}}).
For more results in this direction, the interested readers can refer to \cite{{Grenier-Guo-Nguyen1},{Grenier-Guo-Nguyen2},{Gie-Temam}}
and references therein.

In this work, we will consider the Prandtl equations \eqref{Prandtl-eq}
under Oleinik's monotonicity assumption:
$$
w:=\p_y u>0.
$$
Under this hypothesis, one must further assume $U>0$.
Let us first introduce some weighted Sobolev spaces for later use.
Denoting the vorticity $w:=\p_y u$,
we define the weighted Sobolev space $H^{s, \gamma}$ for $w$ by
$$
H^{s,\gamma}:= \{w: \mathbb{T}\times \mathbb{R}^+ \rightarrow \mathbb{R}: \|w\|_{H^{s,\gamma}}<\infty \},
$$
where the weighted $H^{s,\gamma}$ norm is defined by
$$
\|w\|_{H^{s,\gamma}}^2:=\sum_{|\a|\le s}
  \|(1+y)^{\gamma+\a_2} D^\a w\|_{L^2(\mathbb{T} \times \mathbb{R}^+)}^2,
$$
where $D^\a=\p_x^{\a_1}\p_y^{\a_2}$.
Here, the main idea is adding an extra weight $(1+y)$ for each $y-$derivative.
This corresponds to the weight $\frac{1}{y}$ in the Hardy-type inequality.
Let us define
\begin{equation}\label{B-norm}
\|w(t)\|_{\mathcal{B}^{s,\ga,\si}}^2:=\|w(t)\|_{H^{s,\ga}}^2+\sum_{1\le |\a|\le 2}\|\y^{\si+\a_2}D^\a w(t)\|_{L^\inf}^2,
\end{equation}
and introduce the space
\begin{equation*}
\widetilde{H}^{s,\ga}_{\si,\d_0}:=
\{w:\T \times \mathbb{R}^+ \rightarrow \mathbb{R}|
    \|w(t)\|_{\mathcal{B}^{s,\ga,\si}}^2<+\inf,
    \quad \y^\si w \ge \d_0,\}
\end{equation*}
where $s\ge 4, \ga \ge 1, \si >\ga+\frac{1}{2}, \d_0 \in (0, \frac{1}{2})$.
Now, we can state our main result:

\begin{theorem}\label{Main-Result}
Let $s \ge 4$ be an even integer, $\gamma \ge 1, \sigma> \gamma+\frac{1}{2}$
and $\d_0 \in (0, \frac{1}{2})$.
Suppose the outer flow $U$ satisfies
\begin{equation}\label{U-assumption}
M_U:=\sum_{k=0}^{s/2+1} \underset{0\le t \le T}{\sup} \|\p_t^k U\|_{H^{s-2k+2}(\T)}^2<+\infty.
\end{equation}
Assume that the initial tangential velocity $u_0-U|_{t=0}\in H^{s, \gamma-1}$
and the initial vorticity $w_0:=\p_y u_0 \in \widetilde{H}^{s,\ga}_{\si,2\d_0}$.
Then there exist a times $T=T(s, \gamma, \sigma, \delta_0, \|w_0\|_{\mathcal{B}^{s,\ga,\si}}, M_U)>0$ and
a unique classical solution $(u, v)$ to the Prandtl equations \eqref{Prandtl-eq}-\eqref{o-Bernoulli} such that
\begin{equation*}
\underset{0\le t \le T}{\sup}{\|w(t)\|_{\mathcal{B}^{s,\ga,\si}}^2}
\le {C}_{s,\ga,\si} \{1+\|w_0\|_{\mathcal{B}^{s,\ga,\si}}^8+M_U^4\}<+\infty,
\end{equation*}
and
\begin{equation*}
\underset{\T\times\mathbb{R}^+}{\min}\y^{\si}w(t)\ge \d_0,
\end{equation*}
for all $t\in [0, T]$.
\end{theorem}

\begin{remark}
Defined the space
$$
\begin{aligned}
H^{s,\ga}_{\si,\d_0}:=
&\{w:\T\times \mathbb{R}^+\rightarrow \mathbb{R}:
      \|w\|_{H^{s,\ga}}^2<+\infty,
      \ \sum_{|\a| \le 2}|\y^{\si+\a_2} D^\a w|^2 \le \frac{1}{\d_0^2},
      ~\y^{\si}w\ge \d_0
\},
\end{aligned}
$$
Masmoudi and Wong \cite{Masmoudi-Wong} established the local-in-time well-posedness theory
for the Prandtl equations \eqref{Prandtl-eq} if the initial vorticity $w_0$ belongs to
$H^{s,\ga}_{\si, 2\d_0}$ instead of $\widetilde{H}^{s,\ga}_{\si, 2\d_0}$(required in Theorem \ref{Main-Result}).
In other words, they required the initial vorticity itself, first and second order derivatives
with weight in $L^\inf-$norm should be controlled by $\d_0^{-1}$
rather than being sufficiently large.
\end{remark}

\begin{remark}
When the well-posedness for the Prandtl equations \eqref{Prandtl-eq} in the $H^{4,\ga}$-framework,
Masmoudi and Wong \cite{Masmoudi-Wong} required the condition \eqref{initial-data-2} additionally,
which we do not need in Theorem \ref{Main-Result}.
\end{remark}

We now explain main difficulties of proving Theorem \ref{Main-Result} as well as our
strategies for overcoming them. In order to solve the Prandtl equations \eqref{Prandtl-eq}
in certain $H^s$ Sobolev space, the main difficulty comes from the vertical velocity
$v=-\p_y^{-1}\p_x u$ creates a loss of tangential derivative, so the standard energy
methods can not apply directly.
The main idea of establishing the well-posedness of Prandtl equations \eqref{Prandtl-eq}
is to apply the so-called vanishing viscosity and nonlinear cancellation methods.
To this end, we consider the following approximate system(or regularized Prandtl equations cf.\cite{Masmoudi-Wong}):
\begin{equation}\label{approximate-u}
\left\{
\begin{aligned}
&\p_t \u+\u \p_x \u+\v \p_y \u-\ep^2 \p_x^2 \u-\p_y^2 \u+\p_x p^\ep=0,\\
&\p_x \u+\p_y \v=0,\\
&\u|_{t=0}=u_0,\\
&\u|_{y=0}=\v|_{y=0}=0, \quad \underset{y\rightarrow +\infty}{\lim}\u(t, x, y)=U(t,x),
\end{aligned}
\right.
\end{equation}
for any $\ep>0$. Here the quantities $p^\ep$ and $U$ satisfy a regularized Bernoulli's law:
\begin{equation}\label{Bernoulli-eq}
\p_t U+U\p_x U=\ep^2 \p_x^2 U-\p_x p^\ep.
\end{equation}

We point out that the approximate system \eqref{approximate-u} will turn into
the original Prandtl equations \eqref{Prandtl-eq} as the parameter
$\ep$ tends to zero.
For any $\ep>0$, the local-in-time well-posedness of approximate system \eqref{approximate-u}
is obtained easily in life time $[0, T^\ep]$($T^\ep$ may depends on parameter $\ep$), and hence,
we hope to prove that the solution of \eqref{approximate-u}
in life time $[0, T^\ep]$ will converge to the solution of \eqref{Prandtl-eq} when artificial
viscosity $\ep$ tends to zero.
For this purpose, we need to prove the time of existence $T^\ep$ stays bounded away from zero.
Since the domain considered in this article is periodic, the main part of the boundary layer
will vanish or being stable.
Thus, the main difficulty to prove $T^\ep$ staying bounded away from zero arises from
the vertical velocity $\v$ .
This can be overcame by the nonlinear cancellation methods developed in \cite{Masmoudi-Wong}.
Here we also mention that the readers interested in the vanishing viscosity limit
for incompressible Navier-Stokes equations can refer to
\cite{{Xiao-Xin},{Masmoudi-Rousset},{Masmoudi-Rousset2}} and references therein.

Now, let us explain the main idea to prove $T^\ep$ staying bounded away from zero
and the main novelty to relax to additional conditions \eqref{initial-data-1} and \eqref{initial-data-2}
required in \cite{Masmoudi-Wong}.
First of all, we derive the weighted $L^2$ estimate for $D^\a \w$ for $|\a|\le s$ and $\a_1 \le s-1$.
This works since we are allowed to loss at least one $x-$regularity in these cases.
Secondly, we introduce the quantity $\g_s$(cf.\cite{Masmoudi-Wong}):
\begin{equation}
\g_s:=\p_x^s \w-\frac{\p_y \w}{\w}\p_x^s(\u-U)
\end{equation}
and the weighted norm
\begin{equation}\label{Hg}
\|\w\|_{H^{s,\ga}_g(\T\times \mathbb{R}^+)}^2
:=\|\y^{\ga}\g_s\|_{L^2(\T\times \mathbb{R}^+)}^2
+\sum_{\substack{ |\alpha| \le s \\ \a_1 \le s-1}}
\|\y^{\ga+\a_2} D^\a \w\|_{L^2(\T\times \mathbb{R}^+)}^2,
\end{equation}
provided $\w:=\p_y \u>0$.
As mentioned in \cite{Masmoudi-Wong}, this quantity $\g_s$ can avoid the loss of $x-$derivative
by the nonlinear cancellation; in other words, the quantity $\g_s$ in $L^2$-norm
will have uniform bound independent of $\ep$.
Without the additional condition \eqref{initial-data-1}(required in \cite{Masmoudi-Wong}),
we need to control the quantity $\y^{\si+\a_2}D^{\a}\w(1\le |\a|\le 2)$ in $L^\infty-$norm to close the estimate.
Due to the weight index $\si >\ga+\frac{1}{2}$, it is not easy to apply the quantity $\|\w\|_{H^{s,\ga}_g}$
to control these quantities by the Sobolev inequality.
Define
\begin{equation*}
\|\w(t)\|_{\mathcal{B}^{s,\ga,\si}_g}^2
:=\|\w(t)\|_{H^{s,\ga}_g}^2+\sum_{1\le |\a|\le 2}\|\y^{\si+\a_2}D^\a \w(t)\|_{L^\inf}^2,
\end{equation*}
we may apply the maximum principle of heat equation to control quantities
$\y^{\si+\a_2}D^{\a}\w(1\le |\a|\le 2)$ in $L^\infty-$norm
by $\|\w(t)\|_{\mathcal{B}^{s,\ga,\si}_g}^2$,
initial and boundary data, which can be controlled by the quantity $\|\w\|_{H^{s, \ga}_g}$
after using the Sobolev inequality.
An important remark is that $\|\w(t)\|_{\mathcal{B}_g^{s,\ga,\si}}$ is equivalent to
$\|\w(t)\|_{\mathcal{B}^{s,\ga,\si}}$(see \eqref{XY-norm} and \eqref{YX-norm}),
which has uniform bound on the life span time $[0, T^\ep]$(see Lemma \ref{Existence-Fixed} in Appendix \ref{appendixC}).
Thus, based on the estimates obtained above, we can choose the life span time $T_a$
independent of $\ep$ such the quantity $\|\w(t)\|_{\mathcal{B}^{s,\ga,\si}}$ has
uniform bound on $[0, T_a]$.
Then we can pass the limit $\ep\rightarrow 0^+$ and obtain the existence
and uniqueness of solution to the original Prandtl equations \eqref{Prandtl-eq}
by the approximate system \eqref{approximate-u}.
The weighed norm $\|\cdot\|_{\mathcal{B}_g^{s,\ga,\si}}$ is not only equivalent
to norm $\|\cdot\|_{\mathcal{B}^{s,\ga,\si}}$, but also avoids the loss of tangential
derivatives without the condition \eqref{initial-data-1}.
This is the main novelty in our paper and help us improve the recent result \cite{Masmoudi-Wong}.
But we should point out that the idea, which overcomes the loss of tangential derivative
arising by vertical velocity $v=-\p_y^{-1}\p_x u$, comes from the nonlinear cancellation method
developed in \cite{Masmoudi-Wong}.

The rest of this paper is organized as follows.
In Section \ref{a priori estimate}, one establishes the a priori estimates for
the approximate system \eqref{approximate-u}.
Some useful inequalities and important equivalent relations
will be stated in Appendixs \ref{appendixA} and \ref{appendixB}.
Before we proceed, let us comment on our notation.
Through this paper, all constants $C$ may be different from line to line.
Subscript(s) of a constant illustrates the dependence of the constant, for example,
$C_s$ is a constant depending on $s$ only.
Denote by $\p_y^{-1}$ the inverse of the derivative $\p_y$, i.e.,
$(\p_y^{-1}f)(y):=\int_0^y f(z)dz$.

\section{A priori estimates}\label{a priori estimate}

In this section, we will derive a priori estimates(independent of $\ep$), which are crucial
to prove the local-in-time well-posedness theory of solutions to original Prandtl equations \eqref{Prandtl-eq}.
Denote vorticity $\w:=\p_y \u$, using the the regularized Prandtl equations \eqref{approximate-u},
we find that this vorticity satisfies the following evolution equations:
\begin{equation}\label{re-Prandtl-w}
\left\{
\begin{aligned}
&\p_t \w+\u \p_x \w+\v \p_y \w=\ep^2 \p_x^2 \w+\p_y^2 \w,\\
&\w|_{t=0}=w_0:=\p_y u_0,\\
&\p_y \w|_{y=0}=\p_x p^\ep,
\end{aligned}
\right.
\end{equation}
where the velocity field $(\u, \v)$ is given by
\begin{equation}\label{re-Prandtl-u}
\u(t, x, y):=U(t, x)-\int_y^{+\inf}\w(t, x, \eta)d\eta,
\end{equation}
and
\begin{equation}\label{re-Prandtl-v}
\v(t, x, y):=-\int_0^y \p_x \u(t, x, \eta)d\eta.
\end{equation}
Next, we derive a life existence time $T_a$(independent of $\ep$) such the
quantity $\|\w(t)\|_{\mathcal{B}^{s,\ga,\si}}$ owning a uniform bound.
More precisely, we have the following results.

\begin{theorem}[a priori estimates]\label{main-thereom}
Let $s \ge 4$ be an even integer, $\gamma \ge1, \sigma> \gamma+\frac{1}{2}, \d_0 \in (0, \frac{1}{2})$,
and $\ep \in (0, 1]$, the smooth solution $(\u, \v, \w)$, defined on $[0, T^\ep]$,
to the regularized Prandtl equations \eqref{re-Prandtl-w}-\eqref{re-Prandtl-v}.
Under the assumptions of Theorem \ref{Main-Result}, there exists a time
$T_a:=T_a(s, \ga, \si, \d_0, \|w_0\|_{\mathcal{B}^{s,\ga,\si}}, M_U)>0$
independent of $\ep$ such the following estimates hold on
\begin{equation}\label{eq-o}
\O(t):=\underset{0\le \t \le t}{\sup}{\|\w(\t)\|_{\mathcal{B}^{s,\ga,\si}}^2}
\le {C}_{s,\ga,\si} \{1+\|w_0\|_{\mathcal{B}^{s,\ga,\si}}^8+M_U^4\},
\end{equation}
and
\begin{equation}\label{eq-w}
\underset{\T \times \mathbb{R}^+}{\min}\y^{\si}\w(t,x,y) \ge \d_0,
\end{equation}
for all $t \in [0, \min(T_a, T^\ep)]$.
\end{theorem}

\begin{remark}
After having the results in Theorem \ref{main-thereom} at hand,
we can pass to the limit $\ep \rightarrow 0^+$ in the regularized Prandtl
equations \eqref{approximate-u}$_1$-\eqref{approximate-u}$_4$ and the regularized
Bernoulli's law \eqref{Bernoulli-eq}.
Thus, it is easy to check that the limit functions $(u, v)$ will solve the
original Prandtl equations \eqref{Prandtl-eq} with the Bernoulli's law \eqref{o-Bernoulli}
in the classical sense(cf.\cite{Masmoudi-Wong}).
On the other hand, the uniqueness of Prandtl equations \eqref{Prandtl-eq}
has already been derived in \cite{Masmoudi-Wong} without the condition \eqref{initial-data-1}.
In other words, the local-in-time well-posedness theory of solutions to the Prandtl equations \eqref{Prandtl-eq}
in Theorem \ref{Main-Result} is a direct consequence of Theorem \ref{main-thereom}.
\end{remark}

Throughout this section, for any small constant $\d \in (0, \frac{1}{2})$, we assume
a priori assumption
\begin{equation}\label{w-d}
\y^{\si}\w(t, x, y)\ge \d, \quad \forall (t, x, y)\in [0, T^\ep]\times \T \times \mathbb{R}^+,
\end{equation}
holds on. Let us define
\begin{equation}\label{Og}
\|\w(t)\|_{\mathcal{B}^{s,\ga,\si}_g}^2:=\|\w(t)\|_{H^{s,\ga}_g}^2+Q(t),
\quad {\rm and}\quad
\O_g(t):=\underset{0\le \t \le t}{\sup}\|\w(\t)\|_{\mathcal{B}^{s,\ga,\si}_g}^2,
\end{equation}
where $Q(t)$ is defined by
\begin{equation}\label{Q}
Q(t):=\sum_{1\le |\a|\le 2}\|\y^{\si+\a_2}D^\a \w(t)\|_{L^\inf}^2.
\end{equation}
\subsection{Weighted Energy Estimates}

In this subsection, we will derive the uniform weighted estimates for the vorticity $\w$,
which plays an important role for us to find the uniform existence life time.
Since one order tangential derivative loss is allowed, we may apply the energy method
to establish the weighted estimates for the vorticity $D^\a \w(|\a|\le s, \a_1 \le s-1)$.

\begin{lemma}\label{basic-lemma1}
Under the hypotheses of Theorem \ref{main-thereom}, we have the following estimate:
\begin{equation*}
\begin{aligned}
&\frac{d}{dt}\sum_{\substack{ |\alpha| \le s \\ \a_1 \le s-1}}
\|\y^{\ga+\a_2} D^\a \w\|_{L^2}^2
+\sum_{\substack{ |\alpha| \le s \\ \a_1 \le s-1}}
\|\y^{\ga+\a_2} (\ep \p_x D^\a \w, \p_y D^\a \w)\|_{L^2}^2\\
\le
&C_{s,\ga,\si,\d}\|\p_x^{s+1} U\|_{L^\inf(\T)}^8
+C_{s,\ga,\si,\d}(1+\|\y^{\si+1}\p_y \w\|_{L^\inf}^8+\|\w\|_{H^{s,\ga}_g}^8)\\
&+C_{s,\ga}(1+\|\w\|_{H^{s,\ga}_g})^{s-2}\|\w\|_{H^{s,\ga}_g}^2
+C_s \sum_{k=0}^{s/2}\|\p_t^k \p_x p^\ep\|_{H^{s-2l}(\T)}^2,
\end{aligned}
\end{equation*}
where the positive constants $C_s, C_{s,\ga}$ and $C_{s, \ga, \si, \d}$ are independent of $\ep$.
\end{lemma}

\begin{proof}
Differentiating the vorticity equation \eqref{re-Prandtl-w} with respect to $x~\a_1$ times and $y~\a_2$ times,
and multiplying the resulting equality by $\y^{2\ga+2\a_2}D^\a \w$,
we get after integrating over $\T\times \mathbb{R}^+$,
\begin{equation}\label{l21}
\frac{1}{2}\frac{d}{dt}\|\y^{\ga+\a_2}D^\a \w\|_{L^2}^2
+\ep^2 \|\y^{\ga+\a_2}\p_x D^\a \w\|_{L^2}^2
=J_1+J_2+J_3+J_4,
\end{equation}
where $J_1, J_2, J_3$ and $J_4$ are defined by
\begin{equation*}
\begin{aligned}
&J_1=\int_{\T\times \mathbb{R}^+} \y^{2\ga+2\a_2} D^\a \w \p_y^2 D^\a \w dxdy,\\
&J_2=(\ga+\a_2)\int_{\T\times \mathbb{R}^+} \y^{2\ga+2\a_2-1}\v |D^\a \w|^2dxdy,\\
&J_3=-\sum_{0<\beta \le \a}\binom{\a}{\beta}
      \int_{\T\times \mathbb{R}^+} \y^{2\ga+2\a_2}D^\a \w \cdot D^\beta \u \p_x D^{\a-\beta}\w dxdy,\\
&J_4=-\sum_{0<\beta \le \a}\binom{\a}{\beta}
      \int_{\T\times \mathbb{R}^+} \y^{2\ga+2\a_2}D^\a \w \cdot D^\beta \v \p_y D^{\a-\beta}\w dxdy.
\end{aligned}
\end{equation*}

First of all, integrating by part and applying the Cauchy inequality, we get
\begin{equation*}
J_1\le -\frac{3}{4}\|\y^{\ga+\a_2}\p_y D^\a \w\|_{L^2}^2
    -\int_\T D^\a \w \p_y D^\a \w|_{y=0}dx+C_{s,\ga}\|\w\|_{H^{s,\ga}}^2.
\end{equation*}
Next, integrating by part and applying the divergence-free condition, it follows
\begin{equation*}
|J_2|\le C_{s, \ga}\|\frac{\v}{1+y}\|_{L^\inf}\|\y^{\ga+\a_2}D^\a \w\|_{L^2}^2.
\end{equation*}
Using the divergence-free condition, Sobolev and Hardy inequalities, we get for $\ga \ge 1$
\begin{equation}\label{l2a}
\begin{aligned}
\|\frac{\v}{1+y}\|_{L^\inf}
\le
& \|\frac{\v+y \p_x U}{1+y}\|_{L^\inf}+\|\frac{y \p_x U}{1+y}\|_{L^\inf}\\
\le
&C\{\|\frac{\v+y \p_x U}{1+y}\|_{L^2}+\|\frac{\p_x \v+y \p_x^2 U}{1+y}\|_{L^2}
    +\|\p_y^2\{\frac{\v+y \p_x U}{1+y}\}\|_{L^2}\}
    +\|\frac{y \p_x U}{1+y}\|_{L^\inf}\\
\le
&C(\|\p_x \w\|_{L^2}+\|\p_x(\u-U)\|_{L^2}+\|\p_x^2(\u-U)\|_{L^2})+\|\p_x U\|_{L^\inf(\T)}\\
\le
& C(\|\w\|_{H^{s,\ga}}+\|\p_x U\|_{L^\inf(\T)}),
\end{aligned}
\end{equation}
and hence, $J_2$ can be estimated as
\begin{equation*}
|J_2|\le C_{s,\ga}(\|\p_x U\|_{L^\inf(\T)}+\|\w\|_{H^{s,\ga}})\|\w\|_{H^{s,\ga}}^2.
\end{equation*}

Deal with term $J_3$.
Using the H\"{o}lder inequality, we get for $0<\beta \le \a$
\begin{equation}\label{l22}
\begin{aligned}
&|\int  \y^{2\ga+2\a_2}D^\a \w \cdot D^\beta \u \p_x D^{\a-\beta}\w dxdy|\\
\le
&\|\y^{\ga+\a_2} D^\beta (\u-U) D^{\a+e_1-\beta}\w\|_{L^2} \|\y^{\ga+\a_2} D^\a \w\|_{L^2}\\
&+\|\y^{\ga+\a_2} \p_x^{\beta_1}  U \p_x^{\a_1+1-\beta_1}\p_y^{\a_2}\w\|_{L^2} \|\y^{\ga+\a_2} D^\a \w\|_{L^2}.
\end{aligned}
\end{equation}
Using the Moser and Hardy inequalities, it follows
\begin{equation}\label{l23}
\|\y^{\ga+\a_2} D^\beta (\u-U) D^{\a+e_1-\beta}\w\|_{L^2}
\le C_{s,\ga}\|\w\|_{H^{s,\ga}}^2.
\end{equation}
Applying the Sobolev and Wirtinger inequalities, we get
\begin{equation}\label{l24}
\|\y^{\ga+\a_2} \p_x^{\beta_1}  U \p_x^{\a_1+1-\beta_1}\p_y^{\a_2}\w\|_{L^2}
\le C_s \|\p_x^s U\|_{L^\inf(\T)}\|\w\|_{H^{s,\ga}}.
\end{equation}
Substituting estimates \eqref{l23} and \eqref{l24} into \eqref{l22}, we obtain
\begin{equation*}
|J_3|\le C_{s,\ga}(\|\p_x^s U\|_{L^\inf(\T)}+\|\w\|_{H^{s,\ga}})\|\w\|_{H^{s,\ga}}^2.
\end{equation*}

Deal with the term $J_4$.
Using the H\"{o}lder inequality, we get for $0<\beta \le \a$
\begin{equation}\label{l25}
\begin{aligned}
&|\int_{\T\times \mathbb{R}^+} \y^{2\ga+2\a_2}D^\a \w \cdot D^\beta \v \p_y D^{\a-\beta}\w dxdy|\\
\le
&\|\y^{\ga+\a_2}  D^\beta (\v+y\p_x U)  D^{\a+e_2-\beta}\w\|_{L^2}\|\y^{\ga+\a_2} D^\a \w\|_{L^2}\\
&+\|\y^{\ga+\a_2}  D^\beta (y \p_x U) D^{\a+e_2-\beta}\w\|_{L^2}\|\y^{\ga+\a_2} D^\a \w\|_{L^2}.
\end{aligned}
\end{equation}
Using the Sobolev and Wirtinger inequalities, it follows
\begin{equation}\label{l26}
\|\y^{\ga+\a_2}  D^\beta (y \p_x U) D^{\a+e_2-\beta}\w\|_{L^2}
\le C\|\p_x^{s+1} U\|_{L^\inf(\T)}\|\w\|_{H^{s,\ga}}.
\end{equation}

For $|\a|\le s-1$, we apply the Moser and Hardy inequalities to get for $i=1,2$,
\begin{equation}\label{l27}
\begin{aligned}
&\|\y^{\ga+\a_2}  D^\beta (\v+y\p_x U)  D^{\a+e_2-\beta}\w\|_{L^2}\\
\le
&C\|D^{e_i}(\v+y\p_x U)\|_{H^{s-2,i-2}}\|\w\|_{H^{s-1,\ga}}
 \le C\|\w\|_{H^{s,\ga}}^2.
\end{aligned}
\end{equation}
Similar, for $|\a|=s$ and $\a_1 \le s-1$, we get for $\beta_2 \ge 1$
\begin{equation}\label{l28}
\begin{aligned}
&\|\y^{\ga+\a_2}  D^\beta (\v+y\p_x U)  D^{\a+e_2-\beta}\w\|_{L^2}\\
=
&\|\y^{\ga+\a_2}  \p_x^{\beta_1}\p_y^{\beta_2-1} \p_x(\u- U)\p_x^{\a_1-\beta_1}\p_y^{\a_2-\beta_2}\p_y \w\|_{L^2}
\le C \|\w\|_{H^{s,\ga}}^2,
\end{aligned}
\end{equation}
and for $\beta_2 =0$
\begin{equation}\label{l29}
\begin{aligned}
&\|\y^{\ga+\a_2}  D^\beta (\v+y\p_x U)  D^{\a+e_2-\beta}\w\|_{L^2}\\
=
&\|\y^{\ga+\a_2} \p_x^{\beta_1} (\v+y\p_x U) \p_x^{\a_1-\beta_1}\p_y^{\a_2-1} \p_y \w\|_{L^2}\\
\le
&\|\v+y\p_x U\|_{H^{s-1,-1}}\|\w\|_{H^{s,\ga}}
\le C\|\w\|_{H^{s,\ga}}^2.
\end{aligned}
\end{equation}
Substituting the estimates \eqref{l26}-\eqref{l29} into \eqref{l25}, it follows
\begin{equation*}
|J_4|\le C_s (\|\p_x^{s+1} U\|_{L^\inf(\T)}+\|\w\|_{H^{s,\ga}})\|\w\|_{H^{s,\ga}}^2.
\end{equation*}

Plugging the estimates of $J_1$ through $J_4$ into the equality \eqref{l21}, we get
\begin{equation}\label{l210}
\begin{aligned}
&\frac{1}{2}\frac{d}{dt}\|\y^{\ga+\a_2}D^\a \w\|_{L^2}^2
+\ep^2 \|\y^{\ga+\a_2}\p_x D^\a \w\|_{L^2}^2
+\frac{3}{4}\|\y^{\ga+\a_2}\p_y D^\a \w\|_{L^2}^2\\
\le
&-\int_\T D^\a \w \p_y D^\a \w|_{y=0}dx
+C_{s,\ga}(1+\|\p_x^{s+1} U\|_{L^\inf(\T)}+\|\w\|_{H^{s,\ga}})\|\w\|_{H^{s,\ga}}^2,
\end{aligned}
\end{equation}
which, together with the relation \eqref{b21}, yields directly
\begin{equation}\label{l211}
\begin{aligned}
&\frac{1}{2}\frac{d}{dt}\|\y^{\ga+\a_2}D^\a \w\|_{L^2}^2
+\ep^2 \|\y^{\ga+\a_2}\p_x D^\a \w\|_{L^2}^2
+\frac{3}{4}\|\y^{\ga+\a_2}\p_y D^\a \w\|_{L^2}^2\\
\le
&-\int_\T D^\a \w \p_y D^\a \w|_{y=0}dx
+C_{s,\ga,\si,\d}(1+\|\p_x^{s+1} U\|_{L^\inf(\T)}^8+\|\y^{\si+1}\p_y \w\|_{L^\inf}^8+\|\w\|_{H^{s,\ga}_g}^8).
\end{aligned}
\end{equation}

Note the boundary term above can be estimated as follows(cf.\cite{Masmoudi-Wong}) for $|\a|\le s-1$,
\begin{equation}\label{l212}
|\int_\T D^\a \w \p_y D^\a \w|_{y=0}dx|
\le \frac{1}{12}\|\y^{\ga+\a_2+1}\p_y^2 D^\a \w\|_{L^2}^2
     +C\|\w\|_{H^{s,\ga}_g}^2
\end{equation}
and for $|\a|=s$,
\begin{equation}\label{l213}
|\int_\T D^\a \w \p_y D^\a \w|_{y=0}dx|
\le
\left\{
\begin{aligned}
&\frac{1}{12}\|\y^{\ga+\a_2}\p_y D^\a \w\|_{L^2}^2+H(t),\ \a_2=2k, k\in \mathbb{N};\\
&\frac{1}{12}\|\y^{\ga+\a_2+1}\p_x^{\a_1-1}\p_y^{\a_2+2} \w\|_{L^2}^2+H(t),\ \a_2=2k+1, k\in \mathbb{N};
\end{aligned}
\right.
\end{equation}
where $H(t)=C_s \sum_{l=0}^{s/2}\|\p_t^l \p_x p^\ep \|_{H^{s-2l}(\T)}^2
            +C_{s,\ga}(1+\|\w\|_{H^{s,\ga}_g})^{s-2}\|\w\|_{H^{s,\ga}_g}^2$.
Thus, plugging the estimates \eqref{l212} and \eqref{l213} into \eqref{l211}, and summing over $\a$,
we complete the proof of this lemma.
\end{proof}

Next, define $a^\ep:=\frac{\p_y \w}{\w}$ and $\g_s:=\p_x^s \w-a^\ep \p_x^s(\u-U)$,
we are going to derive the $L^2$ estimate for $\y^\ga \g_s$ by using the standard energy methods.
As mentioned in \cite{Masmoudi-Wong}, this quantity $\g_s$ will avoid the loss of $x-$derivative
by a nonlinear cancellation.
By routine checking, it is easy to justify that the quantity $\g_s$ satisfies the evolution equation(cf.\cite{Masmoudi-Wong})
\begin{equation}\label{eq-gs}
\begin{aligned}
&(\p_t +\u \p_x +\v \p_y-\ep^2 \p_x^2-\p_y^2)\g_s\\
=
&2\ep^2\{\p_x^{s+1}(\u-U)-\frac{\p_x \w}{\w}\p_x^s (\u-U)\}\p_x a^\ep
+2\g_s \p_y a^\ep-\g_1 \p_x^s U\\
&-\sum_{j=1}^{s-1} \binom{s}{j}\g_{j+1} \p_x^{s-j}\u
-\sum_{j=1}^{s-1}\binom{s}{j} \p_x^{s-j}\v \{\p_x^j \p_y \w-a^\ep \p_x^j \w\}\\
&+a^\ep \sum_{j=0}^{s-1}\binom{s}{j}\p_x^j(\u-U)\p_x^{s-j+1}U.
\end{aligned}
\end{equation}
where $\g_k:=\p_x^k \w-a^\ep \p_x^k(\u-U)$.

Now, we are going to derive the following weighted energy estimate for $\g_s$:

\begin{lemma}\label{basic-lemma2}
Under the hypothesis of  Theorem \ref{main-thereom}, we have the following estimate:
\begin{equation*}
\begin{aligned}
&\frac{d}{dt}\|\y^\ga \g_s\|_{L^2}^2
 +\ep^2\|\y^\ga \p_x \g_s\|_{L^2}^2
 +\|\y^\ga \p_y \g_s\|_{L^2}^2\\
\le
&C\|\p_x^{s+1}p^\ep\|_{L^2(\T)}^4+C_{s,\ga, \si, \d}\|\p_x^{s+1}U\|_{L^\inf(\T)}^4
+C_{s,\ga, \si, \d}(1+Q^4(t)+\|\w\|_{H^{s,\ga}_g}^6),
\end{aligned}
\end{equation*}
where the quantity $Q(t)$ is defined in \eqref{Q},
the positive constants $C$ and $C_{s, \ga, \si, \d}$ are independent of $\ep$.
\end{lemma}

\begin{proof}
Multiplying the equation \eqref{eq-gs} by $\y^{2\ga}\g_s$
and integrating over $\T \times \mathbb{R^+}$, we have
\begin{equation}\label{i31a}
\frac{1}{2}\frac{d}{dt}\|\y^\ga \g_s\|_{L^2}^2
+\ep^2\|\y^\ga \p_x \g_s\|_{L^2}^2
=\sum_{i=1}^8 K_i,
\end{equation}
where the terms $K_i(i=1,...,8)$ are defined by
\begin{equation*}
\begin{aligned}
&K_1=\int_{\T\times \mathbb{R}^+} \y^{2\ga}\g_s \p_y^2 \g_s dxdy,\quad
K_2=-\int_{\T\times \mathbb{R}^+} \y^{2\ga}\g_s(\u \p_x \g_s +\v \p_y \g_s)dxdy,\\
&K_3=2\ep^2 \int_{\T\times \mathbb{R}^+} \y^{2\ga}\g_s\{\p_x^{s+1}(\u-U)-\frac{\p_x \w}{\w}\p_x^s(\u-U)\}\p_x a^\ep dxdy,\\
&K_4=2\int_{\T\times \mathbb{R}^+} \y^{2\ga}|\g_s|^2 \p_y a^\ep dxdy,\quad
K_5=\sum_{j=1}^{s-1}\int_{\T\times \mathbb{R}^+} \y^{2\ga}\g_{j+1}\g_s \p_x^{s-j}\u dxdy,\\
&K_6=\int_{\T\times \mathbb{R}^+} \y^{2\ga}\g_1 \g_s \p_x^s U dxdy,\quad
 K_7=\sum_{j=0}^{s-1}\int_{\T\times \mathbb{R}^+} \y^{2\ga} \g_s a^\ep \p_x^j(\u-U) \p_x^{s-j+1}U dxdy,\\
&K_8=\sum_{j=1}^{s-1}\int_{\T\times \mathbb{R}^+} \y^{2\ga}\p_x^{s-j}\v (\p_x^j \p_y \w-a^\ep \p_x^j \w)\g_s dxdy.
\end{aligned}
\end{equation*}

Deal with the term $K_1$.
Integrating by part in the $y-$variable, we get
\begin{equation*}
\begin{aligned}
K_1
&=\int_{\T\times \mathbb{R}^+} \y^{2\ga}\g_s \p_y^2 \g_s dxdy\\
&=\int_\T \g_s \p_y \g_s|_{y=0}dx-\int_{\T\times \mathbb{R}^+} \y^{2\ga}|\p_y \g_s|^2dxdy
  -2\ga \int_{\T\times \mathbb{R}^+} \y^{2\ga-1}\g_s \p_y \g_s dxdy\\
&=-\frac{3}{4}\|\y^\ga \p_y \g_s\|_{L^2}^2-\int_\T \g_s \p_y \g_s|_{y=0}dx+C_\ga \|\y^{\ga-1}\g_s\|_{L^2}^2.
\end{aligned}
\end{equation*}
Due to the definition of $\g_s$ and boundary condition $\p_y \w|_{y=0}=\p_x p^\ep$, it follows
\begin{equation*}
\p_y \g_s|_{y=0}=\p_x^{s+1}p^\ep+\frac{\p_y^2 \w}{\w}\p_x^s U|_{y=0}-\frac{\p_y \w}{\w}\g_s|_{y=0}.
\end{equation*}
Using the H\"{o}lder and trace inequalities, we get
\begin{equation*}
\begin{aligned}
|\int_\T \g_s \p_x^{s+1}p^\ep|_{y=0} dx|
&\le \|\g_s|_{y=0}\|_{L^2(\T)}\|\p_x^{s+1}p^\ep\|_{L^2(\T)}\\
&\le \sqrt{2}\|\g_s\|_{L^2}^{\frac{1}{2}} \|\p_y \g_s\|_{L^2}^{\frac{1}{2}} \|\p_x^{s+1}p^\ep\|_{L^2(\T)}\\
&\le \frac{1}{8}\|\p_y \g_s\|_{L^2}^2+C\|\p_x^{s+1}p^\ep\|_{L^2(\T)}^{\frac{4}{3}}\|\g_s\|_{L^2}^{\frac{2}{3}}.
\end{aligned}
\end{equation*}
Due to $\y^\si \w \ge \d$, it follows $\w|_{y=0}\ge \d$, and hence,
we get, after using $\p_y \w|_{y=0}=\p_x p^\ep$,
\begin{equation*}
|\int_\T \g_s \frac{\p_y \w}{\w}\g_s|_{y=0}dx|
\le \frac{1}{8}\|\p_y \g_s\|_{L^2}^2
     +C_\d\|\p_x p^\ep\|_{L^\inf(\T)}^2\|\g_s\|_{L^2}^2.
\end{equation*}
and
\begin{equation*}
\begin{aligned}
|\int_\T \g_s \frac{\p_y^2 \w}{\w}\p_x^s U|_{y=0}dx|
&\le \d^{-1}\|\p_x^s U\|_{L^\infty(\T)}\|\g_s|_{y=0}\|_{L^2(\T)}\|\p_y^2 \w|_{y=0}\|_{L^2(\T)}\\
&\le 2\d^{-1}\|\p_x^s U\|_{L^\infty(\T)}\|\g_s\|_{L^2}^\frac{1}{2}\|\p_y \g_s\|_{L^2}^\frac{1}{2}
     \|\p_y^2 \w\|_{L^2}^\frac{1}{2}\|\p_y^3 \w\|_{L^2}^\frac{1}{2}\\
&\le \frac{1}{8}\|\p_y \g_s\|_{L^2}^2
     +C_\d\|\p_x^s U\|_{L^\infty(\T)}^{\frac{4}{3}}\|\g_s\|_{L^2}^\frac{2}{3}\|(\p_y^2 \w, \p_y^3 \w)\|_{L^2}^\frac{4}{3}.
\end{aligned}
\end{equation*}
Thus, the term $K_1$ can be estimated as follows
\begin{equation*}
K_1 \le -\frac{1}{2}\|\y^\ga \p_y \g_s\|_{L^2}^2
        +C_{\ga, \d}(1+\|\p_x^{s+1} p^\ep\|_{L^2(\T)}^2+\|\p_x^s U\|_{L^\inf(\T)}^2)(1+\|\w\|_{H^{s,\ga}_g}^2).
\end{equation*}

Deal with the term $K_2$. Integrating by part and applying the divergence-free condition, it follows
\begin{equation*}
|K_2|
=|\int_{\T\times \mathbb{R}^+} \y^{2\ga}\g_s(\u \p_x \g_s +\v \p_y \g_s)dxdy|
\le C_\ga \|\frac{\v}{1+y}\|_{L^\inf}\|\y^\ga \g_s\|_{L^2}^2,
\end{equation*}
which, along with inequality \eqref{b42}, yields directly
\begin{equation*}
|K_2|\le C_{\ga, \si, \d}(1+\|\p_x^s U\|_{L^\inf(\T)}+\|\y^{\si+1}\p_y \w\|_{L^\inf})\|\w\|_{H^{s,\ga}_g}^3.
\end{equation*}

Deal with the term $K_3$.
First of all, using the H\"{o}lder inequality, we get
\begin{equation}\label{i31}
\begin{aligned}
&|\int_{\T\times \mathbb{R}^+} \y^{2\ga}\g_s \frac{\p_x \w}{\w}\p_x^s(\u-U)\p_x a^\ep dxdy|\\
&\le \|\frac{\p_x \w}{\w}\|_{L^\infty}\|\y \p_x a^\ep\|_{L^\inf}
     \|\y^{\ga-1} \p_x^s(\u-U)\|_{L^2}\|\y^{\ga}\g_s\|_{L^2}.
\end{aligned}
\end{equation}
Due to the fact $\p_x a^\ep=\frac{\p_{xy} \w}{\w}-\frac{\p_y \w \p_x \w}{(\w)^2}$,
we get, after using the inequality $\y^\si \w \ge \d$,
\begin{equation}\label{i32}
\|\frac{\p_x \w}{\w}\|_{L^\infty}\le \d^{-1}\|\y^\si \p_x \w\|_{L^\inf}
\end{equation}
and
\begin{equation}\label{i33}
\|\y \p_x a^\ep\|_{L^\inf}
\le \d^{-1}\|\y^{\si+1} \p_{xy}\w\|_{L^\inf}
+\d^{-2}\|\y^{\si+1} \p_{y}\w\|_{L^\inf}\|\y^{\si} \p_{x}\w\|_{L^\inf}.
\end{equation}
Substituting the inequalities \eqref{i32} and \eqref{i33} into \eqref{i31},
and applying inequality \eqref{b31}, we obtain
\begin{equation}\label{i34}
|\int_{\T\times \mathbb{R}^+} \y^{2\ga}\g_s \frac{\p_x \w}{\w}\p_x^s(\u-U)\p_x a^\ep dxdy|
\le C_{\ga,\si,\d}(1+\|\p_x^s U\|_{L^\inf(\T)}^2+Q^2(t))\|\w\|_{H^{s,\ga}_g}^2.
\end{equation}
On the other hand, by virtue of the H\"{o}lder inequality, we get
\begin{equation}\label{i35}
\begin{aligned}
&|\int_{\T\times \mathbb{R}^+} \y^{2\ga}\g_s \p_x^{s+1}(\u-U) \p_x a^\ep dxdy|\\
\le
&\|\y^{\si+1} \w \p_x a^\ep\|_{L^\inf}
 \|\y^{\ga-\si-1}\frac{ \p_x^{s+1}(\u-U)}{\w}\|_{L^2}\|\y^\ga \g_s\|_{L^2}.
\end{aligned}
\end{equation}
Due to the fact $\y^{\si}\w \ge \d$, it follows
\begin{equation}\label{i36}
\|\y^{\si+1} \w \p_x a^\ep\|_{L^\inf}
\le \|\y^{\si+1} \p_{xy} \w\|_{L^\inf}
    +\d^{-1}\|\y^{\si} \p_{x} \w\|_{L^\inf}\|\y^{\si+1} \p_{y} \w\|_{L^\inf}
\end{equation}
By routine checking, it follows
$$
\p_x \g_s=\w \p_y \{\frac{\p_x^{s+1}(\u-U)}{\w}\}-\p_x a^\ep \p_x^s(\u-U),
$$
and hence, we get after applying the Hardy inequality
\begin{equation}\label{i37}
\begin{aligned}
&\|\y^{\ga-\si-1}\frac{ \p_x^{s+1}(\u-U)}{\w}\|_{L^2}\\
&\le C_{\ga, \si}\{\|\frac{\p_x^{s+1}U}{\w|_{y=0}}\|_{L^2(\T)}+
\|\y^{\ga-\si}\p_y\{\frac{\p_x^{s+1}(\u-U)}{\w}\}\|_{L^2}\}\\
&\le C_{\ga, \si, \d}\{\|\p_x^{s+1}U\|_{L^2(\T)}+\|\y^{\ga}\p_x \g_s\|_{L^2}
     +\|\y \p_x a^\ep\|_{L^\inf}\|\y^{\ga-1}\p_x^s(\u-U)\|_{L^2}\}.
\end{aligned}
\end{equation}
Substituting inequalities \eqref{i36} and \eqref{i37} into \eqref{i35},
and applying the inequality \eqref{b31}, it follows
\begin{equation*}
\begin{aligned}
&|\int_{\T\times \mathbb{R}^+} \y^{2\ga}\g_s \p_x^{s+1}(\u-U) \p_x a^\ep dxdy|\\
\le
&\frac{1}{4} \|\y^{\ga}\p_x \g_s\|_{L^2}^2
     +C_{\ga, \si, \d}(1+\|\p_x^{s+1}U\|_{L^2(\T)}^2+Q^2(t))(1+\|\w\|_{H^{s,\ga}_g}^2).
\end{aligned}
\end{equation*}
This and the inequality \eqref{i34} imply  directly
\begin{equation*}
|K_3|\le \frac{1}{2}\ep^2 \|\y^{\ga}\p_x \g_s\|_{L^2}^2
     +C_{\ga, \si, \d}(1+\|\p_x^{s+1}U\|_{L^\inf(\T)}^2+Q^2(t))(1+\|\w\|_{H^{s,\ga}_g}^2).
\end{equation*}

Deal with the term $K_4$.
Indeed, it is easy to get
\begin{equation}\label{i41}
|K_4|=|2\int_{\T\times \mathbb{R}^+} \y^{2\ga}|\g_s|^2 \p_y a^\ep dxdy|
\le C\|\p_y a^\ep\|_{L^\inf}\|\y^\ga \g_s\|_{L^2}^2.
\end{equation}
Due to the fact $\p_y a^\ep=\frac{\p_y^2 \w}{\w}-(\frac{\p_y \w}{\w})^2$, we obtain
\begin{equation*}
\|\p_y a^\ep\|_{L^\inf}
\le \d^{-1}\|\y^\si \p_y^2 \w\|_{L^\inf}
    +\d^{-2}\|\y^\si \p_y \w\|_{L^\inf}^2.
\end{equation*}
This and the inequality \eqref{i41} give immediately
\begin{equation*}
|K_4|\le C_\d(1+Q(t))\|\w\|_{H^{s,\ga}_g}^2.
\end{equation*}

Deal with the term $K_5$. Using the H\"{o}lder inequality, it follows
\begin{equation*}
|K_5|\le C_s \|\p_x^{s-j}\u\|_{L^\inf}\|\y^{\ga} \g_{j+1}\|_{L^2}\|\y^{\ga} \g_s\|_{L^2}.
\end{equation*}
This and the inequalities \eqref{b33} and \eqref{b41} imply directly
\begin{equation*}
|K_5|\le C_{s,\ga,\si}(1+\|\p_x^s U\|_{L^\inf(\T)}^2+Q(t))(1+\|\w\|_{H^{s,\ga}_g}^3).
\end{equation*}

Similarly, it is easy to deduce
\begin{equation*}
|K_6|\le C_{\ga, \d}(1+\|\p_x^s U\|_{L^\inf(\T)}^2+Q(t))\|\w\|_{H^{s,\ga}_g}^2.
\end{equation*}
Deal with the term $K_7$. Using the H\"{o}lder and Hardy inequalities,
we get for $j=0,...,s-1$
\begin{equation*}
\begin{aligned}
|K_7|
&\le C_s \|\p_x^{s-j+1}U\|_{L^\inf(\T)}\|\y a^\ep\|_{L^\inf}
             \|\y^{\ga-1} \p_x^j (\u-U)\|_{L^2}\|\y^\ga \g_s\|_{L^2}\\
&\le C_{s,\ga} \|\p_x^{s+1}U\|_{L^\inf(\T)}\|\y a^\ep\|_{L^\inf}\|\w\|_{H^{s,\ga}_g}^2,
\end{aligned}
\end{equation*}
and hence, using the fact $\|\y a^\ep\|_{L^\inf}\le \d^{-1}\|\y^{\si+1}\p_y \w\|_{L^\inf}$, we get
\begin{equation*}
|K_7|\le C_{s,\ga, \d}(\|\p_x^{s+1}U\|_{L^\inf(\T)}^2+Q(t))\|\w\|_{H^{s,\ga}_g}^2.
\end{equation*}

Deal with the term $K_8$. For the case $j=1$, it is easy to check that
\begin{equation*}
\begin{aligned}
&\int_{\T\times \mathbb{R}^+}  \y^{2\ga}\p_x^{s-1}\v (\p_{xy}\w-a^\ep \p_x \w)\g_s dxdy\\
=&\int_{\T\times \mathbb{R}^+}  \y^{2\ga+1}\frac{\p_x^{s-1}\v+y\p_x^s U}{1+y} (\p_{xy}\w-a^\ep \p_x \w)\g_s dxdy\\
&-\int_{\T\times \mathbb{R}^+}  \y^{2\ga+1}\frac{y\p_x^s U}{1+y} (\p_{xy}\w-a^\ep \p_x \w)\g_s dxdy.
\end{aligned}
\end{equation*}
By virtue of the H\"{o}lder inequality, it follows
\begin{equation}\label{i71}
\begin{aligned}
&|\int_{\T\times \mathbb{R}^+}  \y^{2\ga+1}\frac{\p_x^{s-1}\v+y\p_x^s U}{1+y} (\p_{xy}\w-a^\ep \p_x \w)\g_s dxdy|\\
\le &(\|\y^{\ga+1}\p_{xy}\w\|_{L^\inf}+\|\y a^\ep\|_{L^\inf} \|\y^{\ga}\p_x \w\|_{L^\inf})\\
&\times \|\frac{\p_x^{s-1}\v+y\p_x^s U}{1+y}\|_{L^2} \|\y^{\ga}\g_s\|_{L^2}
\end{aligned}
\end{equation}
and
\begin{equation}\label{i72}
\begin{aligned}
&|\int_{\T\times \mathbb{R}^+}  \y^{2\ga+1}\frac{y\p_x^s U}{1+y} (\p_{xy}\w-a^\ep \p_x \w)\g_s dxdy|\\
\le &(\|\y^{\ga+1}\p_{xy}\w\|_{L^2}+\|\y a^\ep\|_{L^\inf}\|\y^\ga \p_x \w\|_{L^2})\\
&\times \|\p_x^s U\|_{L^\inf(T)}\|_{L^2} \|\y^{\ga}\g_s\|_{L^2}.
\end{aligned}
\end{equation}
Due to the fact $\|\y a^\ep\|_{L^\inf}\le \d^{-1}\|\y^{\si+1} \p_y \w\|_{L^\inf}$,
applying the inequality \eqref{b32} and  Sobolev inequality to
inequalities \eqref{i71} and \eqref{i72}, we get
\begin{equation}\label{i73}
|\int_{\T\times \mathbb{R}^+}  \y^{2\ga}\p_x^{s-1}\v (\p_{xy}\w-a^\ep \p_x \w)\g_s dxdy|
\le C_{\ga, \d}(1+\|\p_x^s U\|_{L^\inf(\T)}^2+Q(t))(1+\|\w\|_{H^{s,\ga}_g}^3).
\end{equation}
On the other hand, by virtue of the H\"{o}lder inequality and estimate \eqref{b42}, we get for $j=2,...,s-1$,
\begin{equation*}
\begin{aligned}
&|\int_{\T\times \mathbb{R}^+} \y^{2\ga}\p_x^{s-j}\v (\p_x^j \p_y \w-a^\ep \p_x^j \w)\g_s dxdy|\\
\le
&(\|\y^{\ga+1}\p_x^j \p_y \w\|_{L^2}+\|\y a^\ep\|_{L^\inf}\|\y^\ga \p_x^j \w\|_{L^2})\\
& \times \|\frac{\p_x^{s-j}\v}{1+y}\|_{L^\inf}\|\y^\ga \g_s\|_{L^2}\\
&\le C_{\ga,\si, \d}(1+\|\p_x^s U\|_{L^\inf(\T)}^2+Q(t))\|\w\|_{H^{s,\ga}_g}^3.
\end{aligned}
\end{equation*}
This and the inequality \eqref{i73} imply
\begin{equation*}
|K_8|\le C_{s,\ga, \si, \d} (1+\|\p_x^s U\|_{L^\inf(\T)}^2+Q(t))(1+\|\w\|_{H^{s,\ga}_g}^3).
\end{equation*}
Substituting the estimates of $K_1$ through $K_8$ into equality \eqref{i31a}, we complete the proof of lemma.
\end{proof}

Based on the estimates obtained in Lemmas \ref{basic-lemma1} and \ref{basic-lemma2}, we have the estimate:
\begin{equation}\label{com-L2}
\begin{aligned}
&\frac{d}{dt}\|\w\|_{H^{s,\ga}_g}^2
 +\ep^2 \sum_{\substack{ |\alpha| \le s \\ \a_1 \le s-1}}\|\y^{\ga+\a_2} \p_x D^\a \w\|_{L^2}^2
 +\ep^2\|\y^\ga \p_x \g_s\|_{L^2}^2\\
&+\sum_{\substack{ |\alpha| \le s \\ \a_1 \le s-1}}\|\y^{\ga+\a_2} \p_y D^\a \w\|_{L^2}^2
 +\|\y^\ga \p_y \g_s\|_{L^2}^2\\
\le
&C_s\{\|\p_x^{s+1}p^\ep\|_{L^2(\T)}^4+\sum_{k=0}^{s/2}\|\p_t^k \p_x p^\ep\|_{H^{s-2k}(\T)}^2\}
 +C_{s,\ga,\si,\d}\|\p_x^{s+1}U\|_{L^\inf(\T)}^4\\
&+C_{s,\ga,\si,\d}\{1+Q^4(t)+\|\w\|_{H^{s,\ga}_g}^{s+4}\}.
\end{aligned}
\end{equation}

Using the regularized Bernoulli's law \eqref{Bernoulli-eq}, we get
\begin{equation}\label{eq-p}
\sum_{k=0}^{s/2}\|\p_t^k \p_x p^\ep\|_{H^{s-2k}(\T)}^2
\le C_s\{1+\sum_{k=0}^{s/2+1} \|\p_t^k U\|_{H^{s-2k+2}(\T)}^2\}^2.
\end{equation}
Using the Sobolev inequality for dimension one, it follows
\begin{equation*}
\|\p_x^{s+1}U\|_{L^\inf(\T)}\le C\|\p_x^{s+1}\|_{H^1(\T)}.
\end{equation*}
This and the inequality \eqref{eq-p} imply
\begin{equation}\label{eq-p-u}
\begin{aligned}
&\|\p_x^{s+1}p^\ep\|_{L^2(\T)}^4+\|\p_x^{s+1}U\|_{L^\inf(\T)}^4
 +\sum_{k=0}^{s/2}\|\p_t^k \p_x p^\ep\|_{H^{s-2k}(\T)}^2\\
&\le  C_s\{1+\sum_{k=0}^{s/2+1} \|\p_t^k U\|_{H^{s-2k+2}(\T)}^2\}^4
 \le  C_s(1+M_U)^4.
\end{aligned}
\end{equation}
Substituting estimate \eqref{eq-p-u} into \eqref{com-L2},
and integrating the resulting inequality over $[0, t]$, it follows
\begin{equation*}
\underset{0\le \t \le t}{\sup}\|\w(\t)\|_{H^{s,\ga}_g}^2
\le
\|w_0\|_{H^{s,\ga}_g}^2+C_{s,\ga,\si,\d}(1+M_U)^4 t
+C_{s,\ga,\si,\d}\int_0^t \{1+Q^4(\t)+\|\w(\t)\|_{H^{s,\ga}_g}^{s+4}\}d\t,
\end{equation*}
or equivalently,
\begin{equation}\label{eq-f1}
\underset{0\le \t \le t}{\sup}\|\w(\t)\|_{H^{s,\ga}_g}^2
\le
\|w_0\|_{H^{s,\ga}_g}^2+C_{s,\ga,\si,\d}(1+M_U)^4 t
+C_{s,\ga,\si,\d} \O_g(t)^s  t,
\end{equation}
for all $s\ge 4$ and $\ga \ge 1$.

\subsection{Weighted $L^\inf$ Estimates for Lower Order Terms}

In this subsection, we will establish the estimate for the quantity $Q(t)$
to close the estimate. Since the weight index $\si>\ga+\frac{1}{2}$, we can not close this $L^\inf$
estimate by the Sobolev inequality and weighted energy estimates directly.
Similar to \cite{Masmoudi-Wong}, we apply the maximum principle of heat equation to control the
quantity $Q(t)$ by its initial data, boundary condition and quantity $\O_g(t)$.
Note that the boundary condition without weight can be controlled by $\O_g(t)$ owning to $\ga \ge 1$.

\begin{lemma}\label{basic-lemma3}
Under the hypotheses of Theorem \ref{main-thereom}, we have the following estimates:
\begin{equation}\label{eq-f2}
\underset{0\le \t \le t}{\sup}Q(\t)\le e^{C_{s,\ga,\si,\d}\{1+M_U+\O_g(t)\}t}
\{\|w_0\|_{\mathcal{B}^{s,\ga,\si}_g}^2+C_{s,\ga,\si,\d}(1+M_U)^4 t+C_{s,\ga,\si,\d} \O_g(t)^{s}t\},
\end{equation}
and
\begin{equation}\label{eq-f3}
\y^{\si}\w(t,x,y)
\ge \y^{\si}w_0(x,y)-Ct (1+M_U+\O(t)).
\end{equation}
\end{lemma}

\begin{proof}
Let us define
$$
I(t):=\sum_{1\le |\a|\le 2}|\y^{\si+\a_2}D^{\a}\w(t)|^2,
$$
similar to \cite{Masmoudi-Wong}, we may check that the quantity $I(t)$ satisfies:
\begin{equation*}\label{241}
\begin{aligned}
\{\p_t +\u \p_x+\v \p_y-\ep^2 \p_x^2-\p_y^2\}I(t)
\le C_{s,\ga, \si, \d}\{1+\|\p_x^s U\|_{L^\inf}^2
   +\|\y^{\si+1}\p_y \w\|_{L^\inf}^2+\|\w\|_{H^{s,\ga}_g}^2\}I(t).
\end{aligned}
\end{equation*}
This and the maximum principle in Lemma \ref{max-p} will give directly for all $t\in [0, T^\ep]$
\begin{equation*}
\|I(t)\|_{L^\inf(\O)}
\le \max\{e^{C_{s,\ga,\si,\d}(1+M_U+\O_g(t))t}\|I(0)\|_{L^\inf(\O)},
           \underset{\t \in [0, t]}{\max}\{e^{C_{s,\ga,\si,\d}(1+M_U+\O_g(t))(t-\t)}\|I(\t)|_{y=0}\|_{L^\inf(\T)}\} \},
\end{equation*}
or equivalently
\begin{equation}\label{241}
\underset{0\le \t \le t}{\sup}Q(\t)
\le e^{C_{s,\ga,\si,\d}\{1+M_U+\O_g(t)\}t}(\|I(0)\|_{L^\inf(\O)}+\|I(t)|_{y=0}\|_{L^\inf(\T)}).
\end{equation}
By virtue of $s \ge 4$, we apply the Sobolev inequality \eqref{Sobolev} to get
\begin{equation*}
\|I(t)|_{y=0}\|_{L^\inf(\T)}\le C\|\w\|_{H^{s,\ga}_g}^2.
\end{equation*}
This and the inequality \eqref{241} yield directly
\begin{equation}\label{242}
\underset{0\le \t \le t}{\sup}Q(\t)\le e^{C_{s,\ga,\si,\d}\{1+M_U+\O_g(t)\}t}(Q(0)+\|\w\|_{H^{s,\ga}_g}^2).
\end{equation}
Substituting the estimate \eqref{eq-f1} into inequality \eqref{242}, we get the inequality \eqref{eq-f2}.

Finally, using the first equation of \eqref{re-Prandtl-w}, it follows
\begin{equation}\label{eq-i}
\begin{aligned}
\|\y^{\si} \p_t\w \|_{L^\inf}
\le
&   \ep^2 \|\y^{\si} \p_x^2 \w\|_{L^\inf}+\|\y^{\si} \p_y^2 \w\|_{L^\inf}
     +\|\y^{\si}\p_x \w\|_{L^\inf}\|\u\|_{L^\inf}\\
&     +\|\y^{\si+1}\p_y \w \|_{L^\inf}\|\frac{\v}{1+y}\|_{L^\inf}\\
\le
&C(1+\|\p_x U\|_{L^\inf(\T)}^2+Q(t)+\|\w\|_{H^{s,\ga}}^2).
\end{aligned}
\end{equation}
Due to the basic fact
\begin{equation*}
\w(t,x,y)-w_0(x,y)=\int_0^t \p_\t \w(\t, x, y)d\t,
\end{equation*}
we get after using \eqref{eq-i}
\begin{equation*}
\begin{aligned}
\y^{\si}\w(t,x,y)
&\ge \y^{\si}w_0(x,y)-\int_0^t \|\y^{\si} \p_\t\w(\t)\|_{L^\inf} d\t\\
&\ge \y^{\si}w_0(x,y)-Ct (1+\underset{0\le \t \le t}{\sup}\|\p_x U(\t)\|_{L^\inf(\T)}^2
+\underset{0\le \t \le t}{\sup}\{\|\w(\t)\|_{H^{s,\ga}}^2+Q(\t)\})\\
&\ge \y^{\si}w_0(x,y)-Ct (1+M_U+\O(t)).
\end{aligned}
\end{equation*}
Therefore, we complete the proof of this lemma.
\end{proof}

From the estimates \eqref{eq-f1}, \eqref{eq-f2} and \eqref{eq-f3}, we have the estimates
\begin{equation}\label{YY-norm}
\O_g(t)
\le 2 e^{C_{s,\ga,\si,\d}\{1+M_U+\O_g(t)\}t}
\{\|w_0\|_{\mathcal{B}^{s,\ga,\si}_g}^2+C_{s,\ga,\si,\d}(1+M_U)^4t+C_{s,\ga,\si,\d}\O_g(t)^{s}t\},
\end{equation}
and
\begin{equation}\label{super-norm}
\y^{\si}\w(t,x,y)
\ge \y^{\si}w_0(x,y)-Ct (1+M_U+\O(t)).
\end{equation}
The advantage of estimates \eqref{YY-norm} and \eqref{super-norm} is  that
the constants $C$ and $C_{s,\ga,\si,\d}$ are independent of the artificial viscosity $\ep$.

\subsection{Proof of Theorem \ref{main-thereom}}

Based on the estimates obtained so far, we can complete the proof of Theorem \ref{main-thereom} in this subsection.
First of all, for any $\ep>0$, we can apply standard energy method to gain the regularity propagates
from the initial data(see estimates \eqref{c31} and \eqref{c32} in Lemma \ref{Existence-Fixed}),
that is to say on $[0, T^\ep]$, we have
\begin{equation*}
\O(T)=\underset{0\le t \le T}{\sup}{\|\w(t)\|_{\mathcal{B}^{s,\ga,\si}}^2}<+\infty.
\end{equation*}
Moreover, we can also get from the initial data that \eqref{w-d} is valid on $[0, T^\ep]$
(possibly by taking $T^\ep$ smaller).
An important remark is that if $\O(T_{ax})<+\infty$, the solution can be
continued on $[0, T_{bx}], T_{bx}>T_{ax}$ with $\O(T_{bx})<+\infty$.
This and the estimates \eqref{eq-o}-\eqref{eq-w} can guarantee that
the solution can be continued on an interval of time independent of $\ep$.
Thus, it suffices to verify the estimates \eqref{eq-o} and \eqref{eq-w}.

For two constants $R$ and $\d$, which will be defined later, we define
\begin{equation}\label{T-define}
T^\ep_*:=\sup\{T\in [0, 1]| \O(t) \le R,\quad \y^{\si} \w(t, x, y)\ge \d,\quad
\forall (t, x, y)\in [0, T]\times \T\times \mathbb{R}^+\}.
\end{equation}

Recall the relations(see \eqref{XY-norm} and \eqref{YX-norm})
\begin{equation*}
\O(t)
\le C_{\ga,\si,\d}\|\p_x^s U\|_{L^\inf(T)}^4+C_{\ga,\si,\d}(1+\O_g(t)^2),
\end{equation*}
and
\begin{equation*}
\O_g(t)\le C_{\ga,\si,\d}(1+\O(t)^2),
\end{equation*}
and hence, we get for all $t\le T^\ep$, after using the inequality \eqref{YY-norm},
\begin{equation*}
\begin{aligned}
\O(t)
\le
&C_{s,\ga,\si}\{1+\O(0)^4+M_U^2+C_{s,\ga,\si,\d}(1+M_U)^8 t^2+C_{s,\ga,\si,\d}(1+\O(t)^2)^{2s}t^2\}\\
&\times e^{C_{s, \ga, \si, \d}\{1+M_U+\O(t)^2\}t}.
\end{aligned}
\end{equation*}
Then, we may conclude for $T \le T_*^\ep$
\begin{equation*}
\begin{aligned}
\O(T)\le
&{C}_{s,\ga,\si}\{1+\O(0)^4+M_U^4+C_{s,\ga,\si,\d}(1+M_U)^8 T+C_{s,\ga,\si,\d}(1+R^2)^{2s}T\}\\
&\times e^{C_{s,\ga,\si,\d}\{1+M_U+R^2\}T}.
\end{aligned}
\end{equation*}
Choose constants $R=8{C}_{s,\ga,\si} \{1+\O(0)^4+M_U^4\}$ and $\d=\frac{\d_0}{2}$,
we get
\begin{equation*}
\O(T_1)\le 4{C}_{s,\ga,\si} \{1+\O(0)^4+M_U^4\}=\frac{R}{2},
\end{equation*}
where  $T_1:=\min\{\frac{{\rm ln} 2}{C_{s,\ga,\si,\d}(1+M_U+R^2)},
\frac{1+\O(0)^4+M_U^4}{2 C_{s,\ga,\si,\d}(1+M_U)^8},
\frac{1+\O(0)^4+M_U^4}{2C_{s,\ga,\si,\d}(1+R^2)^{2s}}\}$.
It follows from \eqref{super-norm}
\begin{equation*}
\underset{\T\times \mathbb{R}^+}{\min}\y^\si \w(t)
\ge \d_0=2\d,\quad t\in[0, T_2].
\end{equation*}
where $T_2:=\min\{T_1, \frac{\d_0}{C(1+M_U+R)}\}$.
Obviously, we conclude that there exists a time $T_2>0$ depending only on
$s,\ga,\si,\d_0, M_U$ and the initial data $\|w_0\|_{\mathcal{B}^{s,\ga,\si}}$(hence independent of $\ep$)
such that for all $T\le \min\{T_2, T^\ep\}$, the estimates \eqref{eq-o} and \eqref{eq-w}
hold on. Of course, it holds that $T_2 \le T^\ep_*$.
Otherwise, our criterion about the continuation of the solution would contradict
the definition of $T^\ep_*$ in \eqref{T-define}.
Then, taking $T_a=T_2$, we obtain the estimate \eqref{eq-w} and close the a priori assumption \eqref{w-d}.
Therefore, we complete the proof of Theorem \ref{main-thereom}.

\appendix

\section{Calculus Inequalities}\label{appendixA}

In this appendix, we will introduce some basic inequality that
be used frequently in this paper. For the proof in detail,
the interested readers can refer to \cite{Masmoudi-Wong}.

\begin{lemma}[Hardy Type Inequalities]
Let function $f:\mathbb{T}\times \mathbb{R}^+ \rightarrow \mathbb{R}$.
\begin{itemize}
\item[(i)] if $\lambda > - \frac{1}{2}$ and $ \lim_{y \to +\infty} f(x,y) = 0$, then
\begin{equation} \label{Hardy1}
    \|(1+y)^\lambda f\|_{L^2 (\mathbb{T}\times \mathbb{R}^+)} \le \frac{2}{2\lambda +1}
    \| (1+y)^{\lambda +1} \partial_y f\|_{L^2 (\mathbb{T}\times \mathbb{R}^+)}.
\end{equation}
\item[(ii)] if $\lambda < - \frac{1}{2}$, then
\begin{equation}\label{Hardy2}
    \|(1+y)^\lambda f\|_{L^2 (\mathbb{T}\times \mathbb{R}^+)}
    \le \sqrt{- \frac{1}{2\lambda +1}}
    \| f|_{y=0}\|_{L^2 (\mathbb{T})}
    - \frac{2}{2\lambda + 1}
    \| (1+y)^{\lambda + 1} \partial_y f \|_{L^2 (\mathbb{T}\times \mathbb{R}^+)}.
\end{equation}
\end{itemize}

\end{lemma}

\begin{lemma}[Sobolev-Type Inequality]
Let the proper function $f: \mathbb{T}\times \mathbb{R}^+ \rightarrow \mathbb{R}$.
Then there exists a universal constant $C>0$ such that
\begin{equation}\label{Sobolev}
\|f\|_{L^\inf(\mathbb{T}\times \mathbb{R}^+)}
\le C\{\|f\|_{L^2(\mathbb{T}\times \mathbb{R}^+)}
+\|\partial_x f\|_{L^2(\mathbb{T}\times \mathbb{R}^+)}
+\|\partial_y^2 f\|_{L^2(\mathbb{T}\times \mathbb{R}^+)}\}.
\end{equation}
\end{lemma}

Next, we state the Morse type inequality that will be used frequently when
we deal with the convective term.
For the sake of brevity, we omit the proof for inequality \eqref{Morse}
since it can be guaranteed by the Sobolev inequality \eqref{Sobolev}.

\begin{lemma}[Morse-Type Inequality]
Let $f$ and $g$ be proper functions, $\ga \in \mathbb{R}$ and an integer $s\ge 3$,
we have for all $|\a+\tilde{\a}|\le s$
\begin{equation}\label{Morse}
\|\y^{\ga+\a+\tilde{\a}}(D^\a f\cdot D^{\widetilde{\a}}g)(t,\cdot)\|_{L^2}
\le C \|f(t)\|_{H^{s,\ga_1}}\|g(t)\|_{H^{s,\ga_2}},
\end{equation}
where $\ga_1, \ga_2 \in \mathbb{R}$ with $\ga_1+\ga_2=\ga$,
and $C>0$ is a universal constant.
\end{lemma}

Finally, let us recall the maximum principle for bounded solutions to parabolic
equations(cf.\cite{Masmoudi-Wong}).

\begin{lemma}[Maximum Principle for Parabolic Equations] \label{max-p}
Let $\ep \ge 0$. If $H \in C([0,T]; C^2 (\T \times \mathbb{R}^+) \cap C^1 ([0,T]; C^0 (\T \times \mathbb{R}^+))$
is a bounded function which satisfies the differential inequality:
$$
    \{\p_t+b_1\p_x+b_2\p_y-\ep^2\p_{xx} - \p_{yy}\} H \le f H \qquad\qquad \text{ in }
    [0,T] \times \T \times \mathbb{R}^+,
$$
where the coefficients $b_1, b_2$ and $f$ are continuous and satisfy
\begin{equation*}
    \left\| \frac{b_2}{1+y} \right\|_{L^\infty ([0,T] \times \T \times \mathbb{R}^+)}  < + \infty
    \qquad \text{ and } \qquad \|f\|_{L^\infty ([0,T] \times \T \times \mathbb{R}^+)}  \le \lambda,
\end{equation*}
then for any $t \in [0,T]$,
\begin{equation*}
    \sup_{\T \times \mathbb{R}^+} H(t) \le \max \{e^{\lambda t} \|H(0)\|_{L^\infty (\T \times \mathbb{R}^+)},
    \max_{\tau \in [0,t]} \{e^{\lambda (t-\tau)} \|H(\tau) |_{y=0}
    \|_{L^\infty (\T)} \} \}.
\end{equation*}
\end{lemma}

\section{Almost Equivalence of Weighted Norms}\label{appendixB}

In this section, we will state some estimates that will be used in section \ref{a priori estimate}.
First of all, we derive the relation between $\g_s$ and $\p_x^s \w$ as follows.

\begin{lemma}
Let $s \ge 4$ be an even integer, $\gamma \ge1, \sigma\ge \gamma+\frac{1}{2}$,
and $\ep \in (0, 1]$, the smooth solution $(\u, \v, \w)$, defined on $[0, T^\ep]$,
to the regularized Prandtl equations \eqref{re-Prandtl-w}-\eqref{re-Prandtl-v}.
There exists a small constant $\d \in (0, 1)$ such that
$\y^\si \w \ge \d, ~ \forall (t, x, y) \in [0 ,T^\ep] \times \mathbb{T} \times \mathbb{R}^+$, then it holds on
\begin{equation}\label{b11}
\|\y^\ga \g_s\|_{L^2}
\le C_{\ga, \d}(1+\|\y^{\si+1}\p_y \w\|_{L^\inf})\|\y^\ga \p_x^s \w\|_{L^2},
\end{equation}
and
\begin{equation}\label{b12}
\|\y^\ga \p_x^s \w\|_{L^2}
\le C_{\ga,\si,\d}(1+\|\p_x^s U\|_{L^\inf(T)}+\|\y^{\si+1}\p_y \w\|_{L^\inf})\|\w\|_{H^{s,\ga}_g},
\end{equation}
where $\g_s:=\p_x^s \w-\frac{\p_y \w}{\w} \p_x^s(\u-U)$.
\end{lemma}

\begin{proof}
Using the definition $g_s^\ep =\p_x^s w^\ep -\frac{\p_y w^\ep}{w^\ep}\p_x^s (u^\es -U)$
and Hardy inequality \eqref{Hardy2}, it follows
\begin{equation*}
\begin{aligned}
\|\y^\ga \g_s\|_{L^2}
&\le \|\y^\ga \p_x^s \w\|_{L^2}+\|\y^\ga \frac{\p_y \w}{\w}\p_x^s (\u-U)\|_{L^2}\\
&\le \|\y^\ga \p_x^s \w\|_{L^2}
     +\|\frac{\y^{\si+1} \p_y \w}{\y^\si \w}\|_{L^\infty}\|\y^{\ga-1}\p_x^s(\u-U)\|_{L^2}\\
&\le \|\y^\ga \p_x^s \w\|_{L^2}
     +C_{\ga, \d}\|\y^{\si+1} \p_y \w\|_{L^\infty}\|\y^{\ga}\p_x^s \w\|_{L^2},
\end{aligned}
\end{equation*}
where we have used $\y^\si \w \ge \d$, and hence, we get
\begin{equation*}
\|\y^\ga \g_s\|_{L^2}
\le C_{\ga, \d}(1+\|\y^{\si+1} \p_y \w\|_{L^\infty})\|\y^{\ga}\p_x^s \w\|_{L^2}.
\end{equation*}
This implies inequality \eqref{b11}.
On the other hand, we get from the definition of $g_s^\ep$ that
\begin{equation}\label{b13}
\|\y^\ga \p_x^s \w\|_{L^2}
\le \|\y^\ga \g_s\|_{L^2}+\|\y^\ga \p_y \w \frac{\p_x^s(\u-U)}{\w}\|_{L^2}
\end{equation}
By routine checking, it follows the relation $\g_s=\w \p_y\{\frac{\p_x^s(\u-U)}{\w}\}$,
and hence using $\u|_{y=0}=0$, we get
$$
\frac{\p_x^s(\u-U)}{\w}=-\frac{\p_x^s U}{\w|_{y=0}}+\int_0^y \frac{\g_s}{\w}d\xi.
$$
This and the condition $\y^\si \w \ge \d$ yield directly
\begin{equation}\label{b14}
\begin{aligned}
&\|\y^\ga \p_y \w \frac{\p_x^s(\u-U)}{\w}\|_{L^2}\\
&\le \d^{-1}\|\p_x^s U\|_{L^\inf(\T)}\|\y^\ga \p_y \w \|_{L^2}
+\|\y^{\si+1} \p_y \w \|_{L^\inf}
 \|\y^{\ga-\si-1}\int_0^y \frac{\g_s}{\w} d\xi\|_{L^2}\\
&\le C_\d\|\p_x^s U\|_{L^\inf(\T)}\|\y^\ga \p_y \w \|_{L^2}
+C_{\ga,\si,\d}\|\y^{\si+1} \p_y \w \|_{L^\inf}\|\y^{\ga} \g_s \|_{L^2}.
\end{aligned}
\end{equation}
Plugging inequality \eqref{b14} into \eqref{b13}, we obtain the inequality \eqref{b12}.
\end{proof}

Based on the inequalities \eqref{b11}-\eqref{b12}, and the definitions of
$H^{s, \ga}$ and $H^{s, \ga}_g$, we can establish the following estimates,
which are important relation for us to obtain the well-posedness for the
Prandtl equations in Sobolev space.

\begin{lemma}
Let $s \ge 4$ be an even integer, $\gamma \ge1, \sigma\ge \gamma+\frac{1}{2}$,
and $\ep \in (0, 1]$, the smooth solution $(\u, \v, \w)$, defined on $[0, T^\ep]$,
to the regularized Prandtl equations \eqref{re-Prandtl-w}-\eqref{re-Prandtl-v}.
There exists a small constant $\d \in (0, 1)$ such that
$\y^\si \w \ge \d, ~ \forall (t, x, y) \in [0 ,T^\ep] \times \mathbb{T} \times \mathbb{R}^+$, then it holds on
\begin{equation}\label{XY-norm}
\O(t)
\le C_{\ga,\si,\d}\|\p_x^s U\|_{L^\inf(T)}^4+C_{\ga,\si,\d}(1+\O_g(t)^2),
\end{equation}
and
\begin{equation}\label{YX-norm}
\O_g(t) \le C_{\ga,\si,\d}(1+\O(t)^2).
\end{equation}
\end{lemma}

\begin{proof}
Using the definition of $\|\cdot\|_{H^{s,\ga}_g}$ and estimate \eqref{b12}, it follows
\begin{equation}\label{b21}
\|\w\|_{H^{s,\ga}}
\le C_{\ga,\si,\d}(1+\|\p_x^s U\|_{L^\inf(T)}+\|\y^{\si+1}\p_y \w\|_{L^\inf})\|\w\|_{H^{s,\ga}_g},
\end{equation}
which along with Cauchy inequality implies directly
\begin{equation*}
\begin{aligned}
\|\w(t)\|_{\mathcal{B}^{s,\ga,\si}}^2
\le
&C_{\ga,\si,\d}(1+\|\p_x^s U\|_{L^\inf(T)}^2+\|\y^{\si+1}\p_y \w\|_{L^\inf}^2)\|\w\|_{H^{s,\ga}_g}^2\\
&+\sum_{1\le |\a|\le 2}\|\y^{\si+\a_2}D^\a \w(t)\|_{L^\inf}^2\\
&\le C_{\ga,\si,\d}\|\p_x^s U\|_{L^\inf(T)}^4+C_{\ga,\si,\d}(1+\|\w(t)\|_{\mathcal{B}_g^{s,\ga,\si}}^4).
\end{aligned}
\end{equation*}
This yields inequality \eqref{XY-norm}.
Similarly, it follows from inequality \eqref{b11} that
\begin{equation}\label{b22}
\|\w\|_{H^{s,\ga}_g}
\le C_{\ga, \si, \d}(1+\|\y^{\si+1}\p_y \w\|_{L^\inf})\|\w\|_{H^{s,\ga}},
\end{equation}
and hence, we obtain
\begin{equation*}
\|\w(t)\|_{\mathcal{B}_g^{s,\ga,\si}}^2 \le C_{\ga, \si, \d}(1+\|\w(t)\|_{\mathcal{B}^{s,\ga,\si}}^4).
\end{equation*}
Then we prove the inequality \eqref{YX-norm}.
\end{proof}

Next, we establish some estimates for the quantity $(\u, \v, \g_k)$ in weighted $L^2-$norm.

\begin{lemma}
Let $s \ge 4$ be an even integer, $\gamma \ge1, \sigma\ge \gamma+\frac{1}{2}$,
and $\ep \in (0, 1]$, the smooth solution $(\u, \v, \w)$, defined on $[0, T^\ep]$,
to the regularized Prandtl equations \eqref{re-Prandtl-w}-\eqref{re-Prandtl-v}.
There exists a small constant $\d \in (0, 1)$ such that
$\y^\si \w \ge \d, ~ \forall (t, x, y) \in [0 ,T^\ep] \times \mathbb{T} \times \mathbb{R}^+$, then it holds on: \\
{\rm(i)}For all $k=0,1,...,s$,
\begin{equation}\label{b31}
\|\y^{\ga-1}\p_x^k(\u-U)\|_{L^2}
\le C_{\ga,\si,\d}(1+\|\p_x^s U\|_{L^\inf(\T)}+\|\y^{\si+1}\p_y \w\|_{L^\inf})\|\w\|_{H^{s,\ga}_g},
\end{equation}
{\rm(ii)}For all $k=0,1,...,s-1$,
\begin{equation}\label{b32}
\|\frac{\p_x^k \v+y \p_x^{k+1}U}{1+y}\|_{L^2}
\le C_{\ga,\si,\d}(1+\|\p_x^s U\|_{L^\inf(\T)}+\|\y^{\si+1}\p_y \w\|_{L^\inf})\|\w\|_{H^{s,\ga}_g},
\end{equation}
{\rm(iii)}For all $k=1,2,...,s$
\begin{equation}\label{b33}
\|\y^\ga \g_k\|_{L^2}
\le C_{\ga,\d}(1+\|\y^{\si+1}\p_y \w\|_{L^\inf})\|\w\|_{H^{s,\ga}_g}.
\end{equation}
\end{lemma}

\begin{proof}
{\rm(i)}It follows from Hardy inequality \eqref{Hardy1} that
$\|\y^{\ga-1}\p_x^k(\u-U)\|_{L^2}\le C_\ga \|\y^{\ga}\p_x^k \w\|_{L^2}$,
and hence inequality \eqref{b31} is a direct consequence of the inequality \eqref{b21}.

{\rm(ii)}Due to the Hardy inequality \eqref{Hardy2} and divergence-free condition, it follows
\begin{equation*}
\|\frac{\p_x^k \v+y \p_x^{k+1}U}{1+y}\|_{L^2}
\le C\|\p_x^k(\p_y \v+\p_x U)\|_{L^2}
\le C\|\p_x^{k+1}(\u-U)\|_{L^2},
\end{equation*}
and hence, we get inequality \eqref{b32} after using  inequality \eqref{b31}.

{\rm(iii)}For the case $s=1,2,...,s-1$, we get after using $\y^\si \w \ge \d$,
\begin{equation*}
\begin{aligned}
\|\y^\ga \g_k\|_{L^2}
&\le \|\y^\ga \p_x^k \w\|_{L^2}
    +\|\frac{\y^{\si+1}\p_y \w}{\y^\si \w}\|_{L^\inf}
     \|\y^{\ga-1}\p_x^k(\u-U)\|_{L^2}\\
&\le \|\y^\ga \p_x^k \w\|_{L^2}
    +C_{\d}\|\y^{\si+1}\p_y \w \|_{L^\inf}\|\y^{\ga-1}\p_x^k(\u-U)\|_{L^2}.
\end{aligned}
\end{equation*}
This and the Hardy inequality yield directly
\begin{equation*}
\|\y^\ga \g_k\|_{L^2}\le C_{\ga, \d}(1+\|\y^{\si+1}\p_y \w \|_{L^\inf})\|\y^{\ga}\p_x^k \w\|_{L^2}.
\end{equation*}
Therefore, we complete the proof of this lemma.
\end{proof}

Finally, we establish some estimates for the quantity $(\u, \v)$ in $L^\inf-$norm.

\begin{lemma}
Let $s \ge 4$ be an even integer, $\gamma \ge1, \sigma\ge \gamma+\frac{1}{2}$,
and $\ep \in (0, 1]$, the smooth solution $(\u, \v, \w)$, defined on $[0, T^\ep]$,
to the regularized Prandtl equations \eqref{re-Prandtl-w}-\eqref{re-Prandtl-v}.
There exists a small constant $\d \in (0, 1)$ such that
$\y^\si \w \ge \d, ~ \forall (t, x, y) \in [0 ,T^\ep] \times \mathbb{T} \times \mathbb{R}^+$, then it holds on:\\
{\rm(i)}For all $k=0,1,...,s-1$,
\begin{equation}\label{b41}
\|\p_x^k \u\|_{L^\inf}
\le C_{\ga, \si,\d}(1+\|\p_x^s U\|_{L^\inf(\T)}+\|\y^{\si+1}\p_y \w\|_{L^\inf})(1+\|\w\|_{H^{s,\ga}_g})
\end{equation}
{\rm(ii)}For all $k=0,1,...,s-2$,
\begin{equation}\label{b42}
\|\frac{\p_x^k \v}{1+y}\|_{L^\inf}
\le C_{\ga,\si,\d}(1+\|\p_x^s U\|_{L^\inf(\T)}+\|\y^{\si+1}\p_y \w\|_{L^\inf})\|\w\|_{H^{s,\ga}_g}.
\end{equation}
\end{lemma}

\begin{proof}
By virtue of the Sobolev inequality \eqref{Sobolev} and estimate \eqref{b31}, it follows
\begin{equation}\label{b43}
\begin{aligned}
\|\p_x^k \u\|_{L^\inf}
&\le C(\|\p_x^k(\u-U)\|_{L^2}+\|\p_x^{k+1}(\u-U)\|_{L^2}+\|\p_x^k \p_y \w\|_{L^2})
     +\|\p_x^k U\|_{L^\inf(\T)}\\
&\le C_{\ga, \si,\d}(1+\|\p_x^s U\|_{L^\inf(\T)}+\|\y^{\si+1}\p_y \w\|_{L^\inf})\|\w\|_{H^{s,\ga}_g}
     +\|\p_x^k U\|_{L^\inf(\T)}.
\end{aligned}
\end{equation}
Due to the fact $U=\int_0^{+\inf} \w dy$, we get
\begin{equation*}
\|U\|_{L^2(\T)}\le C_\ga \|\y^{\ga}\w\|_{L^2},
\end{equation*}
and hence, we get after using the Sobolev and Wirtinger inequalities for $k=0,1,...,s-1$,
\begin{equation}\label{b44}
\|\p_x^k U\|_{L^\inf(\T)}
\le C(\|\p_x^k U\|_{L^2(\T)}+\|\p_x^{k+1} U\|_{L^2(\T)})
\le C_\ga \|\y^{\ga}\w\|_{L^2}+\|\p_x^s U\|_{L^\inf(\T)}.
\end{equation}
Submitting inequality \eqref{b44} into \eqref{b43}, we obtain the inequality \eqref{b41}.
Finally, the inequality \eqref{b42} is  consequence of Sobolev inequality \eqref{Sobolev},
estimates \eqref{b31} and \eqref{b32}.
Thus, we complete the proof of this lemma.
\end{proof}

\section{Existence for the Regularized Prandtl Equations}\label{appendixC}

In this section, we state the local in time well-posedness theory for the
regularized Prandtl equations \eqref{re-Prandtl-w}-\eqref{re-Prandtl-v}. More precisely, we have the following results:

\begin{lemma}\label{Existence-Fixed}
Let $s \ge 4$ be an even integer, $\ga \ge 1, \si>\ga+\frac{1}{2}, \d_0 \in (0, \frac{1}{2})$
and $\ep \in (0, 1]$.
If the vorticity $w_0 \in \widetilde{H}^{s,\ga}_{\si,2\d_0}$,
$U$ and $p^\ep$ are given and satisfy the regularized Bernoulli's law \eqref{Bernoulli-eq}
and the regularity assumption \eqref{U-assumption}, then there exist a time
$$
T^\ep:=T(s,\ga, \si, \d_0, \ep, \|w_0\|_{\mathcal{B}^{s,\ga, \si}}, M_U)>0,
$$
and a solution $\w$, to the regularized vorticity system \eqref{re-Prandtl-w}-\eqref{re-Prandtl-v}, satisfying the estimates:
\begin{equation}\label{c31}
\O(t):=\underset{0\le \t \le t}{\sup}\|\w(\t)\|_{\mathcal{B}^{s,\ga,\si}}^2
\le C_*(1+\|w_0\|_{\mathcal{B}^{s,\ga, \si}}^2)<+\inf,
\end{equation}
and
\begin{equation}\label{c32}
\y^{\si}\w(t,x,y)\ge c_* \d_0,
\end{equation}
for all $(t, x, y)\in [0, T^\ep]\times \T \times \mathbb{R}^+$,
here $C_*, c_*$ are positive constants and $c_*\in (0, 1)$.
\end{lemma}

\begin{proof}
We only establish the a priori estimates \eqref{c31} and \eqref{c32},
and the well-posedness results of the regularized vorticity system \eqref{re-Prandtl-w}-\eqref{re-Prandtl-v}
can be obtained immediately(cf.\cite{Masmoudi-Wong}).

Step 1: $H^{s,\ga}-$estimates. Differentiating the regularized vorticity
equation \eqref{re-Prandtl-w} with differential operator $D^{\a}(|\a|\le s)$, we get
\begin{equation*}
\{\p_t +\u \p_x +\v \p_y-\ep^2 \p_x^2-\p_y^2\}D^\a \w=-[D^\a, \u \p_x]\w-[D^\a, \v \p_y]\w.
\end{equation*}
Multiplying this equation by $\y^{2\ga+2\a_2} D^\a \w$ and integrating over $\T\times \mathbb{R}^+$,
it follows
\begin{equation}\label{c33}
\begin{aligned}
&\frac{1}{2}\frac{d}{dt}\int_{\TR} \y^{2\ga+2\a_2}|D^\a \w|^2 dxdy
+\ep^2 \int_{\TR} \y^{2\ga+2\a_2}|\p_x D^\a \w|^2 dxdy\\
=&(\ga+\a_2)\int_{\TR} \y^{2\ga+2\a_2-1} \v |D^\a \w|^2 dxdy
+\int_{\TR} \y^{2\ga+2\a_2} \p_y^2 D^\a \w \cdot D^\a \w dxdy\\
&-\!\int_{\TR}\!\!\!  \y^{2\ga+2\a_2} [D^\a, \u \p_x]\w \cdot D^\a \w dxdy
-\!\int_{\TR}\!\!\! \y^{2\ga+2\a_2}[D^\a, \v \p_y]\w \cdot D^\a \w dxdy,
\end{aligned}
\end{equation}
where we have used the divergence-free condition.

First of all, we deal with the case $|\a|\le s$ and $\a_1 \le s-1$ in \eqref{c33}.
Similar to \eqref{l210}, we get
\begin{equation}\label{c34}
\begin{aligned}
&\frac{1}{2}\frac{d}{dt}\|\y^{\ga+\a_2}D^\a \w\|_{L^2}^2
+\ep^2 \|\y^{\ga+\a_2}\p_x D^\a \w\|_{L^2}^2
+\frac{3}{4}\|\y^{\ga+\a_2}\p_y D^\a \w\|_{L^2}^2\\
\le
&-\int_\T D^\a \w \p_y D^\a \w|_{y=0}dx
+C_{s,\ga}(1+\|\p_x^{s+1} U\|_{L^\inf(\T)}+\|\w\|_{H^{s,\ga}})\|\w\|_{H^{s,\ga}}^2,
\end{aligned}
\end{equation}
Note the boundary term above can be estimated as follows(cf.\cite{Masmoudi-Wong}) for $|\a|\le s-1$,
\begin{equation}\label{c35}
|\int_\T D^\a \w \p_y D^\a \w|_{y=0}dx|
\le \frac{1}{12}\|\y^{\ga+\a_2+1}\p_y^2 D^\a \w\|_{L^2}^2
     +C\|\w\|_{H^{s,\ga}}^2
\end{equation}
and for $|\a|=s$,
\begin{equation}\label{c36}
|\int_\T D^\a \w \p_y D^\a \w|_{y=0}dx|
\le
\left\{
\begin{aligned}
&\frac{1}{12}\|\y^{\ga+\a_2}\p_y D^\a \w\|_{L^2}^2+G(t),\ \a_2=2k, k\in \mathbb{N};\\
&\frac{1}{12}\|\y^{\ga+\a_2+1}\p_x^{\a_1-1}\p_y^{\a_2+2} \w\|_{L^2}^2+G(t),\ \a_2=2k+1, k\in \mathbb{N};
\end{aligned}
\right.
\end{equation}
where
$$G(t)=C_s \sum_{l=0}^{s/2}\|\p_t^l \p_x p^\ep \|_{H^{s-2l}(\T)}^2
            +C_{s,\ga}(1+\|\w\|_{H^{s,\ga}})^{s-2}\|\w\|_{H^{s,\ga}}^2.$$
Thus, plugging the estimates \eqref{c35} and \eqref{c36} into \eqref{c34},
and summing over $\a$, we get
\begin{equation}\label{c37}
\begin{aligned}
&\frac{d}{dt}\sum_{\substack{ |\alpha| \le s \\ \a_1 \le s-1}}
\|\y^{\ga+\a_2} D^\a \w\|_{L^2}^2
+\sum_{\substack{ |\alpha| \le s \\ \a_1 \le s-1}}
\|\y^{\ga+\a_2} (\ep \p_x D^\a \w, \p_y D^\a \w)\|_{L^2}^2\\
\le
&C_{s,\ga}\|\p_x^{s+1} U\|_{L^\inf(\T)}^2
+C_s \sum_{k=0}^{s/2}\|\p_t^k \p_x p^\ep\|_{H^{s-2l}(\T)}^2
+C_{s,\ga}(1+\|\w\|_{H^{s,\ga}}^s).
\end{aligned}
\end{equation}

Next, we deal with the case $|\a|=s$ and $\a_1=s$ in \eqref{c33},
and hence, it follows
\begin{equation}\label{c38}
\begin{aligned}
&\frac{1}{2}\frac{d}{dt}\int_{\TR} \y^{2\ga}|\p_x^s \w|^2 dxdy
+\ep^2 \int_{\TR} \y^{2\ga}|\p_x^{s+1}  \w|^2 dxdy\\
=&(\ga+\a_2)\int_{\TR} \y^{2\ga-1} \v |\p_x^s \w|^2 dxdy
+\int_{\TR} \y^{2\ga} \p_y^2 \p_x^s \w \cdot \p_x^s \w dxdy\\
&-\int_{\TR} \y^{2\ga} [\p_x^s, \u \p_x]\w \cdot \p_x^s \w dxdy
-\int_{\TR} \y^{2\ga}[\p_x^s, \v \p_y]\w \cdot \p_x^s \w dxdy,
\end{aligned}
\end{equation}
Using the inequality \eqref{l2a}, it follows
\begin{equation}\label{c38a}
\begin{aligned}
&|(\ga+\a_2)\int_{\T\times \mathbb{R}^+} \y^{2\ga-1}\v |\p_x^s \w|^2dxdy|\\
&\le C_{s, \ga}\|\frac{\v}{1+y}\|_{L^\inf}\|\y^{\ga}\p_x^s \w\|_{L^2}^2\\
&\le C_{s, \ga}(\|\p_x U\|_{L^\inf(\T)}+\|\w\|_{H^{s,\ga}})\|\y^{\ga}\p_x^s \w\|_{L^2}^2.
\end{aligned}
\end{equation}
Integrating by part and applying Cauchy inequality, we get
\begin{equation}\label{c39}
\begin{aligned}
&\int_{\TR} \y^{2\ga} \p_y^2 \p_x^s \w \cdot \p_x^s \w dxdy\\
=
&\int_\T \p_y \p_x^s \w \cdot \p_x^s \w|_{y=0}dx
 -\int_{\TR} \y^{2\ga} |\p_y \p_x^s \w|^2 dxdy\\
&-\int_{\TR} \y^{2\ga-1} \p_y \p_x^s \w \cdot \p_x^s \w dxdy\\
\le
& -\frac{3}{4}\int_{\TR} \y^{2\ga} |\p_y \p_x^s \w|^2 dxdy
  +C\int_{\TR} \y^{2\ga-2} |\p_x^s \w|^2 dxdy\\
&+\int_\T \p_y \p_x^s \w \cdot \p_x^s \w|_{y=0}dx.
\end{aligned}
\end{equation}
Using the boundary condition $\p_y \w|_{y=0}=p^\ep$ and Sobolev inequality, it follows
\begin{equation*}
\begin{aligned}
\int_\T \p_y \p_x^s \w \cdot \p_x^s \w|_{y=0}dx
\le
&\sqrt{2}\|\p_x^{s+1} p^\ep\|_{L^2(\T)}
 \|\p_y \p_x^s \w\|_{L^2}^{\frac{1}{2}}\|\p_x^s \w\|_{L^2}^{\frac{1}{2}}\\
\le
&\frac{1}{4} \|\p_y \p_x^s \w\|_{L^2}^2
  +C(\|\p_x^{s+1} p^\ep\|_{L^2(\T)}^2+\|\p_x^s \w\|_{L^2}^2).
\end{aligned}
\end{equation*}
This and the inequality \eqref{c39} yield directly
\begin{equation}\label{c310}
\begin{aligned}
&\int_{\TR} \y^{2\ga} \p_y^2 \p_x^s \w \cdot \p_x^s \w dxdy\\
\le
& -\frac{1}{2}\int_{\TR} \y^{2\ga} |\p_y \p_x^s \w|^2 dxdy
 +C(\|\p_x^{s+1} p^\ep\|_{L^2(\T)}^2+\|\y^{\ga-1}\p_x^s \w\|_{L^2}^2).
\end{aligned}
\end{equation}
By virtue of H\"{o}lder inequality, we get
\begin{equation*}
\begin{aligned}
&|\int_{\TR} \y^{2\ga}[\p_x^s, \u \p_x]\w \cdot \p_x^s \w dxdy|\\
\le
&\sum_{1\le k \le s}C_{s,k}\|\y^{\ga} \p_x^k(\u-U)\p_x^{s+1-k}\w\|_{L^2}\|\y^{\ga} \p_x \w\|_{L^2}\\
&+\sum_{1\le k \le s}C_{s,k}\|\y^{\ga} \p_x^k U \p_x^{s+1-k}\w\|_{L^2}\|\y^{\ga} \p_x \w\|_{L^2}.
\end{aligned}
\end{equation*}
Using the Hardy and Morse type inequalities, it follows for $1\le k \le s$
\begin{equation*}
\|\y^{\ga} \p_x^k(\u-U)\p_x^{s+1-k}\w\|_{L^2}
\le
\|\p_x(\u-U)\|_{H^{s-1,0}}\|\p_x\w\|_{H^{s-1,\ga}}
\le C\|\w\|_{H^{s,\ga}}^2,
\end{equation*}
and applying Wirtinger inequality, it follows
\begin{equation*}
\begin{aligned}
\|\y^{\ga} \p_x^k U \p_x^{s+1-k}\w\|_{L^2}
\le C_s \|\p_x^s U\|_{L^\inf(\T)}\|\w\|_{H^{s,\ga}}.
\end{aligned}
\end{equation*}
Thus, we can conclude the estimate
\begin{equation}\label{c311}
|\int_{\TR} \y^{2\ga}[\p_x^s, \u \p_x]\w \cdot \p_x^s \w dxdy|
\le C_s(\|\p_x^s U\|_{L^\inf(\T)}+\|\w\|_{H^{s,\ga}})\|\w\|_{H^{s,\ga}}^2.
\end{equation}

Applying the H\"{o}lder inequality, it follows
\begin{equation}\label{c312}
\begin{aligned}
&\int_{\TR} \y^{2\ga}[\p_x^s, \v \p_y]\w \cdot \p_x^s \w dxdy\\
\le
&\sum_{1\le k \le s}C_{s,k}\| \y^{\ga}\p_x^k(\v+y\p_x U)\p_x^{s-k}\p_y \w\|_{L^2}
                           \| \y^{\ga}\p_x^s \w\|_{L^2}\\
&+\sum_{1\le k \le s}C_{s,k}\| \p_x^{k+1}U\|_{L^\inf(\T)}\|\y^{\ga+1} \p_x^{s-k}\p_y \w\|_{L^2}
                           \|\y^{\ga}\p_x^s \w\|_{L^2}.
\end{aligned}
\end{equation}
For $1\le k\le s-1$, using the Hardy and Morse type inequalities, it follows
\begin{equation}\label{c313}
\begin{aligned}
&\| \y^{\ga}\p_x^k(\v+y\p_x U)\p_x^{s-k}\p_y \w\|_{L^2} \\
\le
&C\|\p_x (\v+y\p_x U)\|_{H^{s-2,-1}}\|\y \p_{xy}\w\|_{H^{s-2,\ga}}\\
\le
&C\|\w\|_{H^{s,\ga}}^2,
\end{aligned}
\end{equation}
and
\begin{equation}\label{c314}
\begin{aligned}
&\| \y^{\ga}\p_x^s(\v+y\p_x U)\p_y \w\|_{L^2}\\
\le
&\| \y^{\ga-1}\p_x^s(\v+y\p_x U)\|_{L^2}\|\y \p_y \w\|_{L^\inf}\\
\le
&C_\ga \|\y^{\ga}\p_x^{s+1} \w\|_{L^2}\|\w\|_{H^{3,0}}.
\end{aligned}
\end{equation}
Substituting the estimates \eqref{c313} and \eqref{c314} into \eqref{c312},
and using the Cauchy inequality, we get
\begin{equation}\label{c315}
\begin{aligned}
&|\int_{\TR} \y^{2\ga}[\p_x^s, \v \p_y]\w \cdot \p_x^s \w dxdy|\\
\le
&\frac{1}{4}\ep^2 \|\y^{\ga}\p_x^{s+1} \w\|_{L^2}^2
    +C_s \|\p_x^{s+1}U\|_{L^\inf}^2
    +C_{s,\ga,\ep}(1+\|\w\|_{H^{s,\ga}}^4).
\end{aligned}
\end{equation}
Substituting the estimates \eqref{c38a}, \eqref{c310}, \eqref{c311}, \eqref{c315}
into \eqref{c38}, we obtain
\begin{equation*}
\begin{aligned}
&\frac{d}{dt}\int_{\TR} \y^{2\ga} |\p_x^s \w|^2 dxdy
 +\ep^2 \int_{\TR} \y^{2\ga} |\p_x^{s+1} \w|^2 dxdy
 +\int_{\TR} \y^{2\ga} |\p_y \p_x^s \w|^2 dxdy\\
\le
&C_{s}(\|\p_x^{s+1}U\|_{L^\inf(\T)}^2+\|\p_x^{s+1}p^\ep\|_{L^2(\T)}^2)
 +C_{s,\ga,\ep}(1+\|\w\|_{H^{s,\ga}}^4).
\end{aligned}
\end{equation*}
This along with the inequalities \eqref{c37} and \eqref{eq-p-u} yield directly
\begin{equation*}
\begin{aligned}
\frac{d}{dt}\|\w\|_{H^{s,\ga}}^2
+\ep^2\|\p_x \w\|_{H^{s,\ga}}^2+\|\p_y \w\|_{H^{s,\ga}}^2
\le C_{s,\ga}(1+M_U)^2+C_{s,\ga,\ep}\|\w\|_{H^{s,\ga}}^s,
\end{aligned}
\end{equation*}
and hence, we conclude that
\begin{equation}\label{c316}
1+\|\w(t)\|_{H^{s,\ga}}^2
\le \frac{1+\|w_0\|_{H^{s,\ga}}^2}{\{1-\frac{2}{s-2}\max\{C_{s,\ga}(1+M_U)^2,C_{s,\ga,\ep}\}
    \{{1+\|\w_0\|_{H^{s,\ga}}^2}\}^{\frac{s-2}{2}}t\}^{\frac{2}{s-2}}},
\end{equation}
as long as $t<\frac{s-2}{2\max\{C_{s,\ga}(1+M_U)^2,C_{s,\ga,\ep}\}\{{1+\|\w_0\|_{H^{s,\ga}}^2}\}^{\frac{s-2}{2}}}$.

Step 2: $L^\inf-$estimates.
Denote $B_\a:=\y^{\si+\a_2} D^\a \w$ and $I:=\sum_{1\le |\a| \le 2}|B_\a|^2$,
it is easy to check that $I$ satisfies the evolution equation
\begin{equation*}
\begin{aligned}
&(\p_t +\u \p_x +\v \p_y-\ep^2 \p_x^2-\p_y^2)I\\
&=-2\sum_{1\le|\a|\le 2}\{\ep^2 |\p_x B_\a|^2+|\p_y B_\a|^2\}
+2\sum_{1\le|\a|\le 2}\{Q_\a B_\a \p_y B_\a+R_\a |B_\a|^2+S_\a B_\a\},
\end{aligned}
\end{equation*}
where the quantities $Q_\a, R_\a$ and $S_\a$ are given explicitly by
\begin{equation*}
Q_\a:=-\frac{2(\si+\a_2)}{1+y},
\quad R_\a:=\frac{\si+\a_2}{1+y}\v+\frac{(\si+\a_2)(\si+a_2+1)}{(1+y)^2}
\end{equation*}
and
\begin{equation*}
S_a:=
-\sum_{0<\beta \le \a}C_{\a,\beta}(1+y)^{\beta_2}\{D^\beta \u B_{\a-\beta+e_1}
+\frac{D^\beta \v B_{\a-\beta+e_2}}{1+y}\}.
\end{equation*}
By routine checking, we get that
\begin{equation*}
|Q_\a|\le C_\si, \quad
|R_\a|\le C_\si(\|\p_x^s U\|_{L^\inf(\T)}+\|\w\|_{H^{s,\ga}}),
\end{equation*}
and
\begin{equation*}
|S_\a|\le C_\si(\|\p_x^s U\|_{L^\inf(\T)}+\|\w\|_{H^{s,\ga}})
\sum_{0<\beta \le \a}\{|B_{\a-\beta+e_1}|+|B_{\a-\beta+e_2}|\}.
\end{equation*}
Then, we can verify that the quantity $I$ satisfies
\begin{equation*}
(\p_t +\u \p_x +\v \p_y-\ep^2 \p_x^2-\p_y^2)I
\le C_\si \{1+\|\p_x^s U\|_{L^\inf(\T)}+\|\w\|_{H^{s,\ga}}\}I,
\end{equation*}
and hence, we apply the maximum principle in Lemma \ref{max-p} to get
\begin{equation}\label{c317}
\underset{0\le \t \le t}{\sup}Q(\t)
\le e^{C_{\si}\{1+\|\p_x^s U\|_{L^\inf(\T)}+\|\w\|_{H^{s,\ga}}\}t}
(\|I(0)\|_{L^\inf(\O)}+\|I(t)|_{y=0}\|_{L^\inf(\T)}).
\end{equation}
By virtue of $s \ge 4$, we apply the Sobolev inequality \eqref{Sobolev} to get
\begin{equation*}
\|I(t)|_{y=0}\|_{L^\inf(\T)}\le C\|\w\|_{H^{s,\ga}}^2.
\end{equation*}
This and the inequality \eqref{c317} yield directly
\begin{equation}\label{c318}
\underset{0\le \t \le t}{\sup}Q(\t)
\le e^{C_{\si}\{1+\underset{0\le \t \le t}{\sup}\|\p_x^s U(\t)\|_{L^\inf(\T)}
+\underset{0\le \t \le t}{\sup}\|\w(\t)\|_{H^{s,\ga}}\}t}(Q(0)+\|\w\|_{H^{s,\ga}}^2).
\end{equation}

Step 3: Life span time.
Taking $T_1=\frac{\{1-(\frac{1}{2})^{\frac{s-2}{2}}\}(s-2)}
{2\max\{C_{s,\ga}(1+M_U)^2,C_{s,\ga,\ep}\}\{{1+\|w_0\|_{H^{s,\ga}}^2}\}^{\frac{s-2}{2}}}$,
we get by using \eqref{c316}
\begin{equation}\label{c319}
\underset{0\le \t \le T_1}{\sup}\|\w(\t)\|_{H^{s,\ga}}^2\le 2(1+\|w_0\|_{H^{s,\ga}}^2).
\end{equation}
Taking $T_2=\min\{T_1, \frac{{\rm ln} 2}{C_{\si}(3+M_U^{\frac{1}{2}}+\sqrt{2}\|w_0\|_{H^{s,\ga}})}\}$,
it follows from \eqref{c318}
\begin{equation}\label{c320}
\underset{0\le t \le T_2}{\sup}Q(t)\le 4(1+Q(0)+\|w_0\|_{H^{s,\ga}}^2).
\end{equation}
Taking $T_3=\min\{T_1, T_2, \frac{\d_0}{C(7+M_U+4Q(0)+6\|w_0\|_{H^{s,\ga}}^2)}\}$, we get by using \eqref{eq-f3}
\begin{equation}\label{c321}
\y^{\si}\w(t,x,y)\ge \y^{\si}w_0(x,y)-\d_0 \ge 2\d_0-\d_0=\d_0.
\end{equation}
Then, we have chosen the life span time $T_a:=T_3$ such the estimates \eqref{c319}-\eqref{c321} hold on.
Finally, we point out that we can use the local existence results established above to
extend the solution(defined on $[0, T_a]$) step by step to the time interval
$[0, T^\ep]$ such the estimates \eqref{c31}  and \eqref{c32} hold on.
Therefore, we complete the proof of this lemma.
\end{proof}

\section*{Acknowledgements}
Jincheng Gao's research was partially supported by
Fundamental Research Funds for the Central Universities(Grants No.18lgpy66)
and NNSF of China(Grants No.11801586).
Daiwen Huang's research was partially supported by the NNSF of China(Grants No.11631008)
and National Basic Research Program of China 973 Program(Grants No.2007CB814800).
Zheng-an Yao's research was partially supported by NNSF of China(Grant No.11431015).

%\begin{thebibliography}{99}

\end{document}